\documentclass[11pt,a4paper]{article}

\usepackage{psfrag}
\usepackage{verbatim}
\usepackage{amsmath,amssymb,cite}
\usepackage{latexsym}
\usepackage{graphicx}
\usepackage{lscape}
\usepackage{afterpage}
\usepackage[color=green!40]{todonotes}
\usepackage[normalem]{ulem}
\usepackage{dutchcal}

\definecolor{phil}{rgb}{0.43, 0.52, 0.28}

\usepackage{hyperref}
\hypersetup{
    colorlinks=true,
    linkcolor=blue,
    filecolor=magenta,      
    urlcolor=cyan,
    citecolor=red,
}
\DeclareMathOperator*{\argmin}{argmin}
\DeclareMathOperator*{\sgn}{sgn}

\newcommand{\ds}{\displaystyle}
\newcommand{\nexto}{\kern -0.54em}
\newcommand{\dR}{{\rm {I\ \nexto R}}}

\newcommand{\dZ}{{\cal Z \kern -0.7em Z}}
\newcommand{\dC}{{\rm\hbox{C \kern-0.8em\raise0.2ex\hbox{\vrule
height5.4pt width0.7pt}}}}
\newcommand{\dQ}{{\rm\hbox{Q \kern-0.85em\raise0.25ex\hbox{\vrule
height5.4pt width0.7pt}}}}
\newcommand{\proofbox}{\hspace{\fill}{$\Box$}}
\newtheorem{lemma}{Lemma}
\newtheorem{theorem}{Theorem}

\newtheorem{remark}{Remark}
\newtheorem{assumption}{Assumption}

\newenvironment{proof}{Proof.}{\proofbox}

\begin{document}

\author{
C. Yal{\c c}{\i}n Kaya\footnote{Mathematics, UniSA STEM, University of South Australia, Mawson Lakes, S.A. 5095, Australia. E-mail: yalcin.kaya@unisa.edu.au\,.}
\and
Lyle Noakes\footnote{Department of Mathematics and Statistics, The University of Western Australia, Nedlands WA 6009, Australia. Email: Lyle.Noakes@uwa.edu.au \,.}
\and
Philip Schrader\footnote{School of Mathematics, Statistics, Chemistry and Physics, Murdoch University, Murdoch WA 6150, Australia. Email: Phil.Schrader@murdoch.edu.au}
}

\title{\vspace{0mm}\bf Curves of Minimax Spirality}

\maketitle

\begin{abstract} 
We study the problem of finding curves of minimum pointwise-maximum arc-length derivative of curvature, here simply called curves of minimax spirality, among planar curves of fixed length with prescribed endpoints and tangents at the endpoints.  We consider the case when simple bounds (constraints) are also imposed on the curvature along the curve.  The curvature at the endpoints may or may not be specified.  We prove via optimal control theory that the optimal curve is some concatenation of Euler spiral arcs, circular arcs, and straight line segments.  When the curvature is not constrained (or when the curvature constraint does not become active), an optimal curve is only made up of a concatenation of Euler spiral arcs, unless the oriented endpoints lie in a line segment or a circular arc of the prescribed length, in which case the whole curve is either a straight line segment or a circular arc segment, respectively.  We propose numerical methods and illustrate these methods and the results by means of three example problems of finding such curves.
\end{abstract}

\begin{verse} {\bf Key words.} {\sf Minimax spirality, Minimax curvature, Optimal control, Bang--bang control, Singular control, Euler spirals.}
\end{verse}

\begin{verse} {\bf AMS subject classifications.} {\sf Primary 49J15, 49K15\ \ Secondary 65K10, 90C30}
\end{verse}

\pagestyle{myheadings}
\thispagestyle{plain}
\markboth{}{\sf\scriptsize Curves of Minimax Spirality\ \ by\ C.~Y.~Kaya, L.~Noakes, and P.~Schrader}

\section{Introduction}

We are interested in finding a $\cal{C}^2$-curve $z:[0,t_f]\longrightarrow\dR^2$ which minimizes the $L^\infty$-norm of the arc-length derivative $\dot{\kappa}(t)/\Vert \dot z(t)\Vert$ of the (signed) curvature $\kappa(t)$, among all planar curves having fixed length; prescribed endpoints $z_0$ and $z_f$ at $0$ and $t_f$ respectively; prescribed tangents $v_0$ and $v_f$ and (possibly) prescribed curvatures $\kappa_0$ and $\kappa_f$, at the endpoints.  We also consider the case when the curvature is constrained by 
$|\kappa(t)| \le a$, where $a$ is a prescribed positive number.
We call this the problem of finding a {\em curve of minimax spirality}. The derivative of acceleration is widely known as jerk but there does not seem to be a name for the corresponding parametrization independent property of curves. The arc-length derivative of curvature measures the degree to which a curve is spiralling (or unspiralling) in on itself, and so \emph{spirality} seems appropriate.

Since the curvature and spirality are independent of parametrization we will consider only curves which are parametrized with respect to arc length so that $|\kappa(t)| = \Vert \ddot z(t) \Vert$, and the parameter $t_f$ is the length of the whole curve (and therefore $t_f$ is fixed). The problem can then be posed as
\[
\mbox{(P)}\left\{\begin{array}{rl}
\ds\min_{z(\cdot)} &\ \ds\max_{t\in[0,t_f]} |\dot{\kappa}(t)| \\[4mm]
\mbox{s.t.} &\ z(0) = z_0\,,\ z(t_f) = z_f\,,\\[2mm]
  &\ \dot{z}(0) = v_0\,,\ \dot{z}(t_f) = v_f\,,\\[2mm]
  &\ \kappa(0) = \kappa_0\,,\ \kappa(t_f) = \kappa_f\,,\ \ |\kappa(t)| \le a\,, \\[2mm]
   &\ \|\dot{z}(t)\| = 1\,,\mbox{ for a.e. }  t\in[0,t_f]\,,
\end{array}\right.
\]
where $\dot{z} = dz/dt$, $\ddot{z} = d^2z/dt^2$, $\|\cdot\|$ is the Euclidean norm, and $v_0$ and $v_f$ are given such that $\|v_0\| = \|v_f\| = 1$.

\subsection{Related problems and relevance}

The authors of this paper have previously studied {\em curves of minimax curvature}~\cite{KayNoaSch2024}, in which the $L^\infty$-norm of $\kappa$ is minimized instead of the $L^\infty$-norm of $\dot{\kappa}$ in~(P),  with no bounds or end conditions imposed on $\kappa(t)$. Curves of minimax curvature are related by interchange of cost and constraint to the {\em Markov--Dubins problem}~\cite{Dubins1957, Markov1889},  where $t_f$ is minimized under the condition that\linebreak $|\kappa(t)| \le a$. Although initially posed, and in some instances studied, by Andrey Markov in 1889~\cite{KreNud1977, Markov1889}, the Markov--Dubins problem was fully solved by Lester Dubins only in 1957~\cite{Dubins1957}.  Dubins' result characterizes solutions in terms of concatenations of  straight lines and circular arcs: let a straight line segment be denoted by an $S$ and a circular arc of the maximum allowed curvature $a$ (or, turning radius $1/a$) by a $C$.  Then the shortest curve is a concatenation of type $CCC$, or of type $CSC$, or a subset thereof.  The curves of minimax curvature, on the other hand, have been shown in~\cite{KayNoaSch2024} to be of type $CCC$, or $COC$, or $CSC$, or $SOS$, or a subset thereof, where $O$ stands for a full circle, i.e., a loop, and the radius $1/a$ of the circular arcs (including that of the loop) is optimally found.  Although a subarc of type $O$ can be regarded as a special case of a subarc of type $C$, it never appears as part of a solution of the Markov--Dubins problem.

In practical applications, while length minimizing curves usually address economical concerns (such as minimizing the length of railway tracks as originally studied by Markov~\cite{KreNud1977, Markov1889} in the late 19th century, or minimizing the length of tunnels in underground mines~\cite{ChaBraRubTho2012} in the 20th and 21st centuries), the curves minimizing curvature or spirality address the comfort of passengers in railway carriages or in vehicles on roads.  The latter curves have been an active area of research; see, for example, \cite{SadRabFelKha2023, ZboWozBol2021}, where a single Euler spiral is employed/fitted as a {\em track transition curve} (or {\em easement spiral}) joining tracks or roads with different profiles (for example, of type $C$ or $S$), for comfort.  We recall that an {\em Euler spiral}, also known as a {\em clothoid} or a {\em Cornu spiral}, is a curve whose curvature changes linearly with the length of the curve.  We will prove here that, in general, with unconstrained curvature, an $L^\infty$ minimizer of spirality is necessarily a $\cal{C}^2$ concatenation of Euler spirals, rather than just a single Euler spiral fitted to the given end conditions.  We stress that the $\cal{C}^2$-curves found as a solution of Problem~(P) are particularly suitable (i.e., implementable) for the kinds of practical applications we have just mentioned, as these curves would provide a (dynamically) smooth transition from one spiral path to another.

Euler spirals are a special case of the log-aesthetic curves  which have been extensively studied for their desirability as splines (see \cite{Inogu2023} and references therein). The properties which make them desirable are discussed in detail by Levien in \cite{Lev2009} with reference to applications in typography. 
The problem of minimizing the $L^2$-norm of spirality also appears in the context of computer aided geometric design - it seems to have been introduced a little earlier by Moreton \cite{Mor1993}, where the minimizers are called minimum curvature variation curves. A discrete Sobolev gradient approach to computing minimum curvature variation curves is described in \cite{Ren2004}, and existence and convergence results for a continuous gradient flow are proved in \cite{And2020}.

An extension of the Markov--Dubins problem which is somewhat more closely related to Problem~(P) was studied by Sussmann~\cite{Sussmann1997}: the problem of finding the shortest curves with simple (specified) bounds on the derivative of their curvature.  The problem studied by Sussmann differs from Problem~(P) in the following ways: (i)~the $L^\infty$-norm of $\dot{\kappa}$ in~(P) is replaced by the curve length $t_f$, (ii)~the constraint $|\dot{\kappa}(t)| \le 1$ is imposed instead of the minimization of the $L^\infty$-norm of $\dot{\kappa}$ and (iii)~the (state) constraint $|\kappa(t)| \le a$ in~(P) is not imposed.  Sussmann establishes, by using the maximum principle and differential geometric arguments, that the optimal control for the problem in~\cite{Sussmann1997} is of bang--bang type (ruling out bang--singular type of control), i.e., that $\dot{\kappa}$ is a piecewise-constant function switching from $-1$ to $1$ or from $1$ to $-1$.  He also asserts that the number of switchings can be infinite, that is, that the optimal control might exhibit chatter (as in Fuller's phenomenon \cite{Borisov2000}).

Reformulating the problem of finding variational curves as an optimal control problem and then employing a maximum principle so as to obtain a classification of the solution curves, as well as to devise a numerical solution procedure, have long been a powerful approach, going back to the 1960s~\cite{ManSch1969}.  Curves between two oriented points, including splines, have since been extensively studied in this way; see, for example, the earlier works in~\cite{AgwMar1998, BoiCerLeb1991, FreObeOpf1999, McClure1975, OpfObe1988, Sussmann1995, Sussmann1997, SusTan1991} as well as the more recent ones in~\cite{Kaya2017, Kaya2019, KayNoa2013, KayNoaSch2024}.

\subsection{Contribution}

In the current paper, we reformulate Problem~(P) as an optimal control problem as in~\cite{Kaya2017, KayNoa2013}  and use a maximum principle involving not only control but also state constraints.   In the absence of the state constraint $|\kappa(t)| \le a$, we  show that the optimal control $u = \dot{\kappa}$, is either of bang--bang type or is totally singular; in other words, $u(t)$ cannot be partially singular, and so bang--singular type of optimal control is ruled out (Lemma~\ref{lem:tot_singular}).  The bang--bang control function switches between $-b$ and $b$, where the value $b>0$ is also optimally found. In the presence of an active state constraint $|\kappa(t)| \le a$, on the other hand, we  show that the boundary control corresponds to a (partially) singular type of control; in other words, the optimal control $u = \dot{\kappa}$ is in general a concatenation of {\em bang} and {\em boundary} arcs, which may well include straight line segments.

Speaking in geometric terms, we show, in the case when Problem~(P) does not have the curvature constraint $\kappa(t) \le a$ imposed (or that $a$ is large enough so that $\kappa(t) \le a$ never becomes active), that a curve of minimax spirality, i.e., a solution curve for Problem~(P), can be of type $S$, or of type $C$, or a concatenation of (a number of) Euler spirals $\mathcal{S}\cdots\mathcal{S}$, an Euler spiral being denoted $\mathcal{S}$ here.  In the case when the constraint $\kappa(t) \le a$ becomes active, i.e., $\kappa(t) = a$ over some nontrivial subinterval of $[0,t_f]$, we show that a curve of minimax spirality is of type $S$, or of type $C$, or a concatenation of (a number of) $\mathcal{S}$, $C$ and $S$.  A classification statement of the solution curves of Problem~(P) is made in Theorem~\ref{theo:classification}.

To get a numerical solution of an optimal control reformulation of Problem~(P), we propose and formulate two well established approaches.  The first approach directly discretizes the optimal control problem, in this case using Euler's scheme, and employs standard optimization software to solve the resulting large-scale finite-dimensional optimization problem.  The second approach parametrizes the optimal control problem with respect to the lengths of the arcs comprising the solution curve and makes use of the structure of the solution found by the first approach, so as to compute the switching or junction times with a high precision.

We illustrate the classification results by means of numerical (critical) solutions of three example problems.  The first example does not involve a curvature constraint along the curve but it has specified curvatures at the end points.  In the second example, we impose a constraint on the curvature, but leave the end-point curvatures free.  The third example involves a curve joining two circular arcs and considers two separate instances.  The first instance results in five arcs, as opposed to four in the first two examples.  The numerical solution for the second instance exhibits what appears to be chattering (as in Fuller's phenomenon).  We propose a work-around for the chatter resulting in a solution curve of seven arcs.  We graphically verify the maximum principle whenever we can, and provide a detailed discussion of the solutions in each example problem.

The paper is organized as follows.  In Section~\ref{sec:reform_max_princ}, we reformulate Problem~(P) as two equivalent optimal control problems and state a maximum principle involving both control and state constraints.  In Section~\ref{sec:classification}, we provide some preliminary results about the optimal control problem we have and state the main result for a classification of the critical curves.  We present the numerical methods in Section~\ref{sec:num_meth}, as well as the numerical experiments involving the three example problems and their numerical solutions.  Finally, concluding remarks are provided in Section~\ref{sec:conclusion}.

\section{Reformulations and Maximum Principle}
\label{sec:reform_max_princ}

If the distance separating $z_0$ and $z_f$ is more than the fixed length $t_f$ then Problem~(P) has no solution.  In fact, with  $t_f = \|z_0-z_f\|$, a feasible solution may still not exist.  Therefore we make the following assumption.

\begin{assumption}
$t_f > \|z_0-z_f\|$.
\end{assumption}

\subsection{Reformulations}
\label{reformulations}

Problem~(P) is nonsmooth because of the $\max$ operator appearing in its objective functional.  On the other hand, Problem~(P) can be transformed into a smooth variational problem by using a standard technique from nonlinear programming, in the same way it was also done in \cite{KayNoa2013, KayNoaSch2024}:
\[
\mbox{(P1)}\left\{\begin{array}{rl}
\ds\min_{z(\cdot)} &\ b \\[2mm]
\mbox{s.t.} &\ z(0) = z_0\,,\ z(t_f) = z_f\,,\\[2mm]
    &\ \dot{z}(0) = v_0\,,\ \dot{z}(t_f) = v_f\,,\\[2mm]
    &\ \kappa(0) = \kappa_0\,,\ \kappa(t_f) = \kappa_f\,,\ \ |\kappa(t)| \le a\,, \\[2mm]
    &\ |\dot{\kappa}(t)| \le b\,,\ \ \ \|\dot{z}(t)\| = 1\,,\mbox{ for a.e. }  t\in[0,t_f]\,,
\end{array}\right.
\]
where the bound $b \ge 0$ is a new optimization variable of the problem.  We re-iterate that our aim is to find a $\cal{C}^2$-curve $z$ solving Problem~(P1).

\begin{remark}  \label{rem:observation} \rm
One has the solution $b = 0$ if and only if the curve joining $z_0$ and $z_f$ is either (i)~a straight line segment or (ii)~a circular arc (we are considering only $\cal{C}^2$-curves, so a concatenation of these is not admissible).  This is possible only in the respective case when (i)~$v_0 = v_f$ and $v_0$ and $v_f$ are colinear with $z_f - z_0$ and $\|z_f - z_0\| = t_f$ or (ii)~$v_0$ and $v_f$ are the tangents at the initial and terminal points of a circular arc of radius greater than or equal to $1/a$ and of length $t_f$.
\proofbox
\end{remark}

In the rest of this paper we will assume that $b > 0$, as there is nothing more to say for the simpler case of $b = 0$.

\begin{assumption} \label{assump:b_positive}
$b > 0$.
\end{assumption}

Problem~(P) can equivalently be cast as an optimal control problem as in \cite{Kaya2017, KayNoaSch2024} as follows. Let\linebreak $z(t) := (x(t), y(t))\in\dR^2$, with $\dot{x}(t) := \cos\theta(t)$ and $\dot{y}(t) := \sin\theta(t)$, where $\theta(t)$ is the angle the velocity vector $\dot{z}(t)$ of the curve $z(t)$ makes with the $x$-axis.  These definitions verify that $\|\dot{z}(t)\| = 1$.  Moreover,
\[
\|\ddot{z}\|^2 = \ddot{x}^2+\ddot{y}^2 = \dot{\theta}^2\,.
\]
Therefore, $|\dot{\theta}(t)|$ is nothing but the curvature.  In fact, $\dot{\theta}(t)$ itself, which can be positive or negative, is referred to as the {\em signed curvature}.  For example, consider a vehicle travelling along a circular path.  If $\dot{\theta}(t) > 0$ then the vehicle travels in the counter-clockwise direction, i.e., it {\em turns left}, and if $\dot{\theta}(t) < 0$ then the vehicle travels in the clockwise direction, i.e., it {\em turns right}.  If $\dot{\theta}(t) = 0$ then the vehicle travels along a straight line.

We note that $\kappa(t) = \dot{\theta}(t)$.  Suppose that the directions at the points $z_0$ and $z_f$ are denoted by the angles $\theta_0$ and $\theta_f$, respectively.  The curvature minimizing problem~(P1), or equivalently Problem~(P), can then be re-written as a ({\em parametric}) {\em optimal control problem}, where the objective functional is the maximum spirality, namely $b$ is the {\em parameter}, $x$, $y$, $\theta$ and $\kappa$ are the {\em state variables}, and $u$ is the {\em control variable}\,:
\[
\mbox{(Pc)}\left\{\begin{array}{rll}
\ds\min_{u(\cdot)} &\ \ds b & \\[2mm]
\mbox{s.t.} &\ \dot{x}(t) = \cos\theta(t)\,, & x(0) = x_0\,,\ 
              x(t_f) = x_f\,, \\[2mm] 
  &\ \dot{y}(t) = \sin\theta(t)\,, & y(0) = y_0\,,\ 
              y(t_f) = y_f\,,\\[2mm] 
  &\ \dot{\theta}(t) = \kappa(t)\,, & \theta(0) = \theta_0\,,\ \theta(t_f) = \theta_f\,,\\[2mm]
  &\ \dot{\kappa}(t) = u(t)\,, & \kappa(0) = \kappa_0\,,\ \kappa(t_f) = \kappa_f\,,\ \ |u(t)|\le b\,,  \\[2mm]
  &\ |\kappa(t)| \le a\,, &\mbox{for a.e. }  t\in[0,t_f]\,.
\end{array}\right.
\] 
Here, $|\kappa(t)| \le a$ can equivalently be written as $\kappa(t) - a \le 0$ and $-\kappa(t) - a \le 0$ which constitute the (smooth) {\em state constraints} of the problem.

In what follows, we study the optimality conditions for a re-formulation of Problem~(Pc) in a classical form. Let us redefine the control variable such that $u(t) := \beta(t)\,v(t)$, where $\beta(t) := b$ and so $\dot{\beta}(t) := 0$, for a.e. $t\in[0,t_f]$.  Here, the constant function $\beta$ is introduced in order to make the optimal control problem a rather standard one, without a parameter constraining the control, to which standard (non-parametric) formulations of the maximum principle are applicable.  Problem~(Pc), or equivalently Problem~(P), can then be re-written as the following {\em state- and control-constrained optimal control problem}\,:
\[
\mbox{(OC)}\left\{\begin{array}{rll}
\ds\min_{v(\cdot)} &\ \ds \int_0^{t_f} \beta(t)\,dt  & \\[4mm]
\mbox{s.t.} &\ \dot{x}(t) = \cos\theta(t)\,, & x(0) = x_0\,,\ 
              x(t_f) = x_f\,, \\[2mm] 
  &\ \dot{y}(t) = \sin\theta(t)\,, & y(0) = y_0\,,\ 
              y(t_f) = y_f\,,\\[2mm] 
  &\ \dot{\theta}(t) = \kappa(t)\,, & \theta(0) = \theta_0\,,\ 
              \theta(t_f) = \theta_f\,,\\[2mm]
  &\ \dot{\kappa}(t) = \beta(t)\,v(t)\,, & \kappa(0) = \kappa_0\,,\ \kappa(t_f) = \kappa_f\,, \\[2mm]
  &\ \dot{\beta}(t) = 0\,, & |v(t)|\le 1\,, \\[2mm]
  &\ \kappa(t) - a \le 0\,, & \\[2mm]
  &\ -\kappa(t) - a \le 0\,, &\mbox{for a.e. }  t\in[0,t_f]\,.
\end{array}\right.
\]

\subsection{Maximum principle for Problem~(OC)}
\label{sec:max_principle}

In this section, we state the maximum principle by using the {\em direct adjoining approach} from \cite[Theorem~4.1]{HarSetVic1995}, where the state constraints are adjoined to the Hamiltonian function ``directly''.  We note that the authors of \cite{HarSetVic1995} designate Theorem~4.1 in their paper as an ``Informal Theorem'' since they say that it has not been proved fully (for the general case of pure state, and mixed control and state, constraints).  However, they also point to Maurer's paper \cite{Maurer1979} as an exception where Theorem~4.1 has been proved for the case of pure state and pure control constraints, which is precisely our setting.  

We start by defining the extended/augmented {\em Hamiltonian function} for Problem~(OC) as
\begin{eqnarray}  \label{aug_Hamiltonian}
H(x,y,\theta,\kappa,\beta,v,\lambda,\mu_1,\mu_2) &:=& \lambda_0\,\beta
+ \lambda_1\,\cos\theta + \lambda_2\,\sin\theta + \lambda_3\,\kappa + \lambda_4\,\beta\,v + \lambda_5 \cdot 0 \nonumber \\
&& +\ \mu_1\,(\kappa - a) + \mu_2\,(-\kappa - a)\,,
\end{eqnarray}
where $\lambda_0\geq0$ is a scalar (multiplier) parameter, $\lambda:[0,t_f]\rightarrow\mathbb{R}^5$ is the {\em adjoint variable} vector with $\lambda(t):=(\lambda_1(t),\dots,\lambda_5(t))$, and $\mu_1,\mu_2:[0,t_f]\rightarrow\mathbb{R}$ are the {\em state constraint multipliers}. For brevity, we use the following notation,
\[
H[t] := H(x(t),y(t),\theta(t),\kappa(t),\beta(t),v(t),\lambda(t),\mu_1(t),\mu_2(t))\,.
\]
The adjoint variables are required to satisfy
\begin{subequations}
\begin{eqnarray}
&& \dot{\lambda}_1(t) := -H_x[t] = 0\,, \label{adjoint1} \\[1mm]
&& \dot{\lambda}_2(t) := -H_y[t] = 0\,, \label{adjoint2} \\[1mm]
&& \dot{\lambda}_3(t) := -H_\theta[t] = \lambda_1(t)\,\sin\theta(t) - \lambda_2(t)\,\cos\theta(t)\,, \label{adjoint3} \\[1mm]
&& \dot{\lambda}_4(t) := -H_\kappa[t] = -\lambda_3(t) - \mu_1(t) + \mu_2(t)\,,\ \ 
\label{adjoint4} \\[1mm]
&& \dot{\lambda}_5(t) := -H_\beta[t] = -\lambda_0 - \lambda_4(t)\,v(t)\,,\ \ \lambda_5(0) = 0\,,\ \lambda_5(t_f) = 0\,, \label{adjoint5} 
\end{eqnarray}
\end{subequations}
where $H_x = \partial H / \partial x$, etc.  

The maximum principle in~\cite[Theorem~4.1]{HarSetVic1995} for Problem~(OC) can be stated as follows.  Suppose that $x,y,\theta,\kappa,\beta\in W^{1,\infty}(0,t_f;\dR)$
and $v\in L^\infty(0,t_f;\dR)$ solve Problem~(OC).  Then
there exist a number $\lambda_0\ge0$, piecewise absolutely continuous adjoint variable vector $\lambda$, $i=1,2,3,4,5$, and piecewise continuous state constraint multipliers $\mu_i$, $i=1,2$, such that $(\lambda_0,\lambda(t), \mu_1(t), \mu_2(t)) \neq \bf0$, for every $t\in[0,t_f]$, and, in addition to the state differential equations and other constraints given in Problem~(OC) and the adjoint
differential equations \eqref{adjoint1}--\eqref{adjoint5}, the following hold for a.e. $t\in[0,t_f]$:  The optimal control is a minimizer of the extended Hamiltonian; namely,
\begin{equation}  \label{control} 
v(t)\in\argmin_{|w|\le 1} H(x(t),y(t),\theta(t), \kappa(t),\beta(t),w,\lambda_0,\lambda(t),\mu_1(t),\mu_2(t))\,,
\end{equation}
the state constraint multipliers satisfy the {\em complementarity conditions},
\begin{subequations}
\begin{align}
 \mu_1(t)\geq 0, \qquad \mu_1(t)\,(\kappa(t)-a)=0\,,  \label{eq:mu1} \\[1mm]
 \mu_2(t)\geq 0, \qquad \mu_2(t)\,(\kappa(t)+a)=0\,,  \label{eq:mu2}
\end{align}
\end{subequations}
and, since $H$ does not depend on $t$ explicitly,
\begin{equation}  \label{H_const}
H[t] = h\,,  
\end{equation}
where $h$ is some real constant.

\begin{remark} \rm
Define the state vector as $\xi(t) := (x(t), y(t), \theta(t), \kappa(t), \beta(t))$.  Suppose that the functional in Problem~(OC) contains a term involving the initial and terminal states, say $\varphi(\xi(0),\xi(t_f))$.  The maximum principle requires that the differential equation $\dot{\lambda} = -\partial H / \partial \xi$ holds with boundary conditions determined as follows. If $\xi_i(0)$ is specified (i.e., fixed), then $\lambda_i(0)$ is free; otherwise, one must have $\lambda_i(0) = -\partial\varphi(\xi(0),\xi(t_f))/\partial \xi_i(0)$.  Similarly, if $\xi_i(t_f)$ is specified, then $\lambda_i(t_f)$ is free; if $\xi_i(t_f)$ is not specified, one has $\lambda_i(t_f) = \partial\varphi(\xi(0),\xi(t_f))/\partial \xi_i(t_f)$.  For these reasons, since $\xi_i(0)$ and $\xi_i(t_f)$ are specified (numbers), for $i = 1,\ldots 4$, $\lambda_i(0)$ and $\lambda_i(t_f)$ are free in \eqref{adjoint1}--\eqref{adjoint4}; and, since $\xi_5(0)$ and $\xi_5(t_f)$, that is, $\beta(0)$ and $\beta(t_f)$, are not specified, given that
$\varphi(\xi(0),\xi(t_f)) = 0$ in Problem~(OC),
one has $\lambda_5(0) = -\partial\varphi(\xi(0),\xi(t_f))/\partial \xi_5(0) = 0$ and $\lambda_5(t_f) = \partial\varphi(\xi(0),\xi(t_f))/\partial \xi_5(t_f)= 0$ in~\eqref{adjoint5}.
\proofbox
\end{remark}

\begin{remark} \rm
For brevity we imposed $-a \le \kappa(t) \le a$, with $a$ constant.  This can be generalized to (the time-varying) $\underline{a}(t) \le \kappa(t) \le \overline{a}(t)$, where $\underline{a} : [0,t_f] \to \dR$ and $\overline{a} : [0,t_f] \to \dR$, $\underline{a}(t) \le \overline{a}(t)$ for all $t\in[0,t_f]$, are continuously differentiable~\cite{HarSetVic1995}.  With the more general constraint, the transversality conditions \eqref{eq:mu1}--\eqref{eq:mu2} would be written down in exactly the same way, with the equalities simply replaced by\ \ $\mu_1(t)\,(\kappa(t)-\overline{a}(t))=0$\ \ and\ \ $\mu_2(t)\,(\kappa(t)+\underline{a}(t))=0$, respectively.
\proofbox
\end{remark}

Using the definition in \eqref{aug_Hamiltonian} and the assumption that $b>0$, \eqref{control} can be concisely written as
\begin{equation}  \label{control2}
v(t)\in\argmin_{|w|\le 1}\ \lambda_4(t)\,w\,,
\end{equation}
which yields the optimal control as
\begin{equation}  \label{control3}
v(t) = \left\{\begin{array}{ll}
\ \ 1\,, & \mbox{if}\  \lambda_4(t) < 0\,, \\[3mm]
-1\,, & \mbox{if}\ \lambda_4(t) > 0\,, \\[3mm]
\mbox{undetermined}\,, & \mbox{if}\ \lambda_4(t) = 0\,.
\end{array}\right.
\end{equation}

Note that since both state constraints cannot be satisfied at the same time, it follows from \eqref{eq:mu1}--\eqref{eq:mu2} that if $\mu_1 > 0$, i.e., if the constraint $\kappa(t) \le a$ is active (which means that $\kappa(t) = a$), then $\mu_2 = 0$.  Likewise, if $\mu_2 > 0$ then $\mu_1 = 0$.

The ODEs in \eqref{adjoint1}--\eqref{adjoint2} imply that $\lambda_1(t) = \overline{\lambda}_1$ and $\lambda_2(t) = \overline{\lambda}_2$ for all $t\in[0,t_f]$, where $\overline{\lambda}_1$ and $\overline{\lambda}_2$ are constants.  Define new constants
\begin{equation}  \label{eq:rhophi}
\rho := \sqrt{\overline{\lambda}_1^2 +
  \overline{\lambda}_2^2}\,,\qquad 
\tan\phi := \frac{\overline{\lambda}_2}{\overline{\lambda}_1}\,.
\end{equation}
Then, using \eqref{eq:rhophi} and \eqref{eq:mu1}--\eqref{eq:mu2}, \eqref{aug_Hamiltonian} can be re-written as
\begin{equation}  \label{Hamiltonian2}
H[t] = b\,\lambda_0 + \rho\,\cos(\theta(t) - \phi) + \lambda_3(t)\,\kappa(t) + b\,\lambda_4(t)\,v(t)\,,
\end{equation}
where we have also used $\beta(t) = b$, for all $t\in[0,t_f]$, and \eqref{adjoint3} can be re-written as
\begin{equation}  \label{adjoint3a}
\dot{\lambda}_3(t) = \rho\,\sin(\theta(t) - \phi)\,.
\end{equation}
Furthermore, \eqref{Hamiltonian2} and \eqref{H_const} give
\begin{equation}  \label{H_const2}
\lambda_3(t)\,\kappa(t) + b\,\lambda_4(t)\,v(t) + \rho\,\cos(\theta(t) - \phi) + b\,\lambda_0 =: h\,.
\end{equation}

Next, we recall the terminology on the types of optimal control for Problem~(OC).  The control $v(t)$ to be chosen for the case when $\lambda_4(t)=0$ for a.e.\ $t\in[\zeta_1,\zeta_2]\subseteq[0,t_f]$ is referred to as {\em singular control} over $[\zeta_1,\zeta_2]$, because \eqref{control2} does not yield any further information.  On the other hand, when $\lambda_4(t)\neq 0$ for a.e.\ $t\in[\zeta_1,\zeta_2]\subseteq[0,t_f]$, i.e., it is possible to have $\lambda_4(t)=0$ only for isolated values of $t$, the control $v(t)$ is said to be {\em nonsingular} over $[\zeta_1,\zeta_2]$.  

It should be noted that, if $\lambda_4(\tau)=0$ only at an isolated point $\tau$, the optimal control at this isolated point can be chosen as $v(\tau) = -1$ or $v(\tau) = 1$, conveniently.  Therefore, if the control $v(t)$ is nonsingular, it will take on either the value $-1$ or $1$, the bounds on the control variable.  In this case, the control $v(t)$ is referred to as {\em bang--bang}.  The trajectory segment generated by $v(t) = 1$ or $v(t) = -1$ is referred to as a {\em bang arc}.  Since the sign of $\lambda_4(t)$ determines the value of the optimal control $v(t)$, $\lambda_4$ is referred to as the {\em switching function}.

If a state constraint is active for $t\in[\zeta_1,\zeta_2]\subseteq[0,t_f]$, then the control $v(t)$ is referred to as {\em boundary control} over $[\zeta_1,\zeta_2]$.  In the case for our particular problem, the boundary control $v(t) = 0$, since $\dot{\kappa}(t) = 0$, for $t\in[\zeta_1,\zeta_2]\subseteq[0,t_f]$.

\begin{remark} \label{rem:C2_curve} \rm
By the maximum principle elaborated above, one concludes that a solution curve for Problem~(OC) is in general a concatenation of bang--bang (or nonsingular), singular and boundary arcs.  By looking at the differential equations in Problem~(OC), we observe that, if $u(t)$ is a concatenation of (finitely many) bang--bang, singular and boundary control (in other words, $u$ is a piecewise-continuous function), then $\kappa$ will be of class $\cal{C}^0$ (to be more precise, $W^{1,\infty}$), $\theta$ class $\cal{C}^1$, and the curve $z$, such that $z(t) = (x(t), y(t))$, class $\cal{C}^2$ (to be more precise, $W^{3,\infty}$, which readily contains $\cal{C}^2$), as intended in Problems~(P), (P1) or (OC).
\proofbox
\end{remark}

\section{Classification of Critical Curves}
\label{sec:classification}

\begin{lemma}[Singularity With Inactive Curvature Constraints] \label{lem:singular_inactive}
Suppose that the curvature constraints do not become active and that the optimal control $v(t)$ for Problem~{\em (OC)} is singular, over an interval $[\zeta_1,\zeta_2]\subseteq[0,t_f]$.  Then $\kappa(t) = 0$, i.e., $\theta(t)$ is constant, and it follows that $v(t) = 0$, for a.e.\ $t\in[\zeta_1,\zeta_2]$.
\end{lemma}
\begin{proof}
Suppose that the constraints on the curvature do not become active and that the optimal control $v(t)$ is singular, i.e., $\lambda_4(t)=0$, and so $\dot\lambda_4(t)=0$, for a.e.\ $t\in[\zeta_1,\zeta_2]\subseteq[0,t_f]$.  Since the constraints on the curvature do not become active, one has that $\mu_1(t) = \mu_2(t) = 0$, for a.e.\ $t\in[\zeta_1,\zeta_2]$.  Therefore, from \eqref{adjoint4}, $\lambda_3(t)=0$  and  $\dot\lambda_3(t)=0$ for a.e.\ $t\in[\zeta_1,\zeta_2]$.  The constant $\rho$ cannot be zero, as otherwise $\lambda_i(t) = 0$, for $i = 1,\ldots,5$, and $\lambda_0 = 0$, violating the maximum principle.  Therefore, from~\eqref{adjoint3a}, $\sin(\theta(t) - \phi) = 0$, which implies that $\theta(t)$ is constant, i.e., $\dot{\theta}(t) = \kappa(t) = 0$, and subsequently $v(t) = 0$, for a.e. $t\in[\zeta_1,\zeta_2]$.
\end{proof}

Recall from optimal control terminology that a trajectory segment where a state constraint is active is called a {\em boundary arc}.  The following lemma states that the optimal control along a boundary arc of Problem~(OC) is the same as that along a singular arc.

\begin{lemma}[Singularity With Active Curvature Constraints] \label{lem:singular_active}
Suppose that a curvature\linebreak constraint of Problem~{\em (OC)} is active over an interval $[\zeta_1,\zeta_2]\subseteq[0,t_f]$.  Then the resulting optimal control along the boundary arc is $v(t) = 0$, the same as that along a singular arc.
\end{lemma}
\begin{proof}
Since a curvature constraint is active, one has that $\kappa(t) = a$ or $\kappa(t) = -a$, and that $\dot{\kappa}(t) = 0$ for a.e.\ $t\in[\zeta_1,\zeta_2]$.  Then, from the fourth state differential equation in Problem~(OC) with $b>0$, one gets $v(t) = 0$, the same as that along a singular as can be deduced from~\eqref{control3}.
\end{proof}

\begin{remark}[Types of Solution Arcs] \label{rem:types} \rm
Lemma~\ref{lem:singular_inactive} asserts that if the curvature is not constrained, then the singular arc, i.e, the curve generated by singular control in some subinterval, can only be a straight line segment. Lemma~\ref{lem:singular_active} on the other hand states that, if the curvature is constrained (and a constraint is active), then the boundary arc is a circular arc.  We note that, although Lemma~\ref{lem:singular_active} asserts that the controls along the singular and boundary arcs are the same, the respective arcs, or the trajectories, themselves are in general different.

If the optimal control is nonsingular, i.e., one has a bang arc, over some nontrivial interval $[\eta_1,\eta_2]\subseteq[0,t_f]$, then, by \eqref{control3}, $\dot\kappa(t)$ is either $b$ or $-b$, i.e., $\dot\kappa(t) = \pm b$, resulting in $\kappa(t) = \pm b\,t + c$, with $c$ another constant, and $\theta(t) = \pm b\,t^2/2 + c\,t + d$, with $d$ yet another constant, for all $t\in[\eta_1,\eta_2]$.  The arc arising from solving the ODEs in Problem~(OC), given
 by
\begin{equation}  \label{spiral}
(x(t),y(t)) = (x(\eta_1),y(\eta_1)) + (\int_{\eta_1}^t \cos\theta(t)\,dt, \int_{\eta_1}^t\sin\theta(t)\,dt)\,,
\end{equation}
for all $t\in[\eta_1,\eta_2]$, is an {\em Euler spiral}, whose curvature $\kappa(t)$ changes linearly with $t$, the current length of the arc.  The Euler spiral is also referred to as a {\em clothoid} or a {\em Cornu spiral}.  We will refer to \eqref{spiral} simply as a {\em spiral}, for brevity, in the rest of the paper.  The integrals appearing in~\eqref{spiral} with $\theta(t)$ quadratic in $t$ are called the {\em Fresnel integrals}.
\proofbox
\end{remark}

Using Lemmas~\ref{lem:singular_inactive} and \ref{lem:singular_active}, we can rewrite the optimal control in \eqref{control3} as
\begin{equation} \label{v}
v(t) = -\sgn(\lambda_4(t))\,, \mbox{ a.e. } t\in[0,t_f]\,. 
\end{equation}

The problems that yield a solution with $\lambda_0 = 0$, and the solutions themselves, are referred to as {\em abnormal} in the optimal control theory literature. In the case of abnormal solutions the conditions obtained from the maximum principle are independent of the objective functional $\int_0^{t_f} \beta(t)\,dt$ in Problem~(OC) and therefore not sufficiently informative.  The problems that yield $\lambda_0 > 0$, and their solutions, are referred to as {\em normal}.  Lemma~\ref{lem:normality} below states that we can rule out abnormality for Problem~(OC).

\begin{lemma}[Normality of Solutions] \label{lem:normality}
Problem~{\em (OC)} is normal, i.e., one can set $\lambda_0 := 1$.
\end{lemma}
\begin{proof}
For contradiction purposes, suppose that $\lambda_0 = 0$ and recall Assumption~\ref{assump:b_positive}, which states that $b>0$. We exclude the case when $b=0$ as it has already been discussed in Remark~\ref{rem:observation}.  Then, from~\eqref{adjoint5} and \eqref{v}, $\dot{\lambda}_5(t) = |\lambda_4(t)| \ge 0$, for every $t\in[0,t_f]$, which, along with the boundary conditions $\lambda_5(0) = \lambda_5(t_f) = 0$, implies that $\lambda_5(t) = \lambda_4(t) = 0$ for every $t\in[0,t_f]$.  In other words, the optimal control is singular over $[0,t_f]$.  Then, by Lemma~\ref{lem:singular_inactive} or \ref{lem:singular_active}, for the respective cases of unconstrained and (active) constrained curvature, $\kappa(t)$ is constant for every $t\in[0,t_f]$, implying that $b = 0$, which is a contradiction.  Therefore $\lambda_0 > 0$.  Then, without loss of generality, one can set $\lambda_0 = 1$.
\end{proof}

Suppose that the curve $(x(t),y(t))$ is made up of $s$ arcs, with each arc defined as the curve segment $(x(t),y(t))$ for $t\in[t_{i-1},t_i]$, where $0 = t_0 < t_1 < \cdots <t_{s-1} < t_s = t_f$.

\begin{lemma}[Switching Function With Inactive Curvature Constraints]  \label{lem:lambda4_DE}
Suppose that the curvature constraints do not become active (or that the curvature in Problem~(OC) is unconstrained). Then the switching function $\lambda_4$ for Problem~{\em (OC)} solves the differential equation
\begin{equation} \label{eq:costate} 
\dddot{\lambda}_4(t) = -\kappa(t) \left(\kappa(t)\,\dot{\lambda}_4(t) + b\,|\lambda_4(t)| - b + h\right)\,,
\end{equation}
where $\kappa(t) = -b\,\sgn(\lambda_4(t))\,t + c_i$, with $c_i$ a constant along the $i$th arc, for all $t\in[0,t_f]$, with $c_i$ chosen in such a way that $\kappa$ is continuous.
\end{lemma}
\begin{proof}
Differentiating both sides of \eqref{adjoint4} twice (with $\mu_1(t) = \mu_2(t) \equiv 0$ since the curvature constraints remain inactive), using \eqref{adjoint3a} and the third ODE in Problem~(OC), one obtains
\begin{equation} \label{eqnA} 
\dddot{\lambda}_4(t) = -\ddot{\lambda}_3 = -\rho\,\dot{\theta}(t)\,\cos(\theta(t) - \phi) = -\rho\,\kappa(t)\,\cos(\theta(t) - \phi)\,.
\end{equation}
Using \eqref{adjoint4}, \eqref{v} and $\lambda_0 = 1$ in \eqref{H_const2}, one gets
\[
\rho\,\cos(\theta(t) - \phi) = \kappa(t)\,\dot{\lambda}_4(t) + b\,|\lambda_4(t)| - b + h\,.
\]
Substituting this into the right-hand side of \eqref{eqnA} and rearranging give \eqref{eq:costate}.  Finally, the solution of the fourth ODE in Problem~(OC), after substituting $\beta(t) = b$ and $v(t)$ in \eqref{v}, yields $\kappa(t) = -b\,\sgn(\lambda_4(t))\,t + c_i$, with $c_i$ a constant along the $i$th arc, for all $t\in[0,t_f]$, with $c_i$ chosen in such a way that $\kappa$ is continuous by Remark~\ref{rem:C2_curve}.
\end{proof}

The following lemma asserts that singularity cannot happen partially; in other words, the bang and singular arcs cannot exist simultaneously in a solution trajectory if neither of the curvature constraints becomes active and there is no chattering (as in Fuller's phenomenon).

\begin{lemma}[Total Singularity] \label{lem:tot_singular}
Suppose that the curvature constraints do not become active (or that the curvature in Problem~(OC) is unconstrained).  If the optimal control $v(t)$ for Problem~{\em (OC)} is singular over an interval $[\zeta_1,\zeta_2]\subseteq[0,t_f]$, and if there is no chattering, then it is singular over $[0,t_f]$.
\end{lemma}
\begin{proof}
Suppose that the optimal control $v(t)$ is singular over $[\zeta_1,\zeta_2]\subseteq[0,t_f]$, i.e., (with $b>0$) $\lambda_4(t) = 0$, for all $t\in[\zeta_1,\zeta_2]$.  By Equations \eqref{adjoint3}--\eqref{adjoint4}, $\dot{\lambda}_4$ and $\ddot{\lambda}_4$ are continuous. Then one also has that $\dot{\lambda}_4(t) = \ddot{\lambda}_4(t) = 0$, and, by Lemma~\ref{lem:singular_inactive}, $\kappa(t) = 0$, for all $t\in[\zeta_1,\zeta_2]$.  Suppose that a switching occurs at $t = \zeta_2$, from the singular arc to a nonsingular arc, which has no chattering.  Then, in the nonsingular arc, for $(\zeta_2, \zeta_2^+)$, i.e., along a ``tiny'' initial segment of the nonsingular arc, $\lambda_4(t)$ is indeed nonzero, and we have $\kappa(t) = -b\,\sgn(\lambda_4(t))\,t + c_i$, for $(\zeta_2, \zeta_2^+)$, by Lemma~\ref{lem:lambda4_DE}.  Moreover, since $\kappa$ is continuous (along the overall trajectory), also by Lemma~\ref{lem:lambda4_DE}, and since $\lambda_4(\zeta_2) = 0$ and $\kappa(\zeta_2) = 0$, we conclude that $c_i = 0$. 
Consequently, the ODE~\eqref{eq:costate} 
in Lemma~\ref{lem:lambda4_DE}, with $\kappa(t) = -b\,\sgn(\lambda_4(t))\,t$ substituted, becomes
\[
\dddot{\lambda}_4(t) = b\,\sgn(\lambda_4(t)) \left(-b\,\sgn(\lambda_4(t))\,\dot{\lambda}_4(t)\,t + b\,|\lambda_4(t)| - b + h\right) t\,.
\]
This is a third-order dynamical system with $(\lambda_4, \dot{\lambda}_4, \ddot{\lambda}_4) = (0,0,0)$ constituting an equilibrium point.  Therefore, ``forward'' solution of 
this ODE with the initial conditions $\lambda_4(\zeta_2) = \dot{\lambda}_4(\zeta_2) = \ddot{\lambda}_4(\zeta_2) = 0$
yields $\lambda_4(t) = 0$ for all $t\in[\zeta_2,\zeta_2^+)$, extending the interval over which $v(t)$ is singular to $[\zeta_1,\zeta_2^+)$.  Via similar arguments, one has that $\lambda_4(\zeta_1) = \dot{\lambda}_4(\zeta_1) = \ddot{\lambda}_4(\zeta_1) = \kappa(\zeta_1) = 0$, and so ``backward'' solution of the above ODE yields $\lambda_4(t) = 0$ for all $t\in(\zeta_1^-,\zeta_1]$, extending the interval over which $v(t)$ is singular. Therefore, switching cannot happen at either endpoint of $[\zeta_1,\zeta_2]$, making the singularity interval $[0,t_f]$.
\end{proof}

Combining the discussion in Remark~\ref{rem:observation}, Remark~\ref{rem:types} incorporating Lemmas~\ref{lem:singular_inactive} and \ref{lem:singular_active}, and Lemma~\ref{lem:tot_singular} and \eqref{v}, one gets the following theorem on the types of critical solutions one obtains.

\begin{theorem}[Classification of critical curves] \label{theo:classification}
If the curvature is unconstrained ($a = \infty$) then, unless there is chattering, an optimal curve for Problem~(P) is (i)~a straight line segment or (ii)~a circular arc or (iii)~a concatenation of spirals.  If the curvature is constrained then an optimal curve for Problem~(P) is a concatenation of a subset of spirals, circular arcs and straight line segments.
\end{theorem}

\subsection{Unspecified curvatures at endpoints}

Suppose that the curvatures at the endpoints are not specified, i.e., $\kappa(0)$ and $\kappa(t_f)$ are free.  Then the only difference the optimality conditions in this case have, compared with the case when $\kappa(0)$ and $\kappa(t_f)$ are specified (i.e., fixed), is that one has to impose the transversality conditions
\begin{equation}  \label{transversality}
\lambda_4(0) = 0\quad\mbox{and}\quad \lambda_4(t_f) = 0\,,
\end{equation}
on the differential equation in~\eqref{adjoint4}.  Although, as expected, the solution curves with the free and fixed curvatures at the endpoints are in general different, the lemmas and Theorem~\ref{theo:classification} stated above remain unchanged.

\section{Numerical Methods and Experiments}
\label{sec:num_meth}

In this section, we implement numerical methods to find an approximate solution to Problem~(P), employing (i) {\em direct discretization techniques}, which in the literature have extensively been studied~\cite{AltBaiLemGer2013, Betts2020, DonHag2001, DonHagMal2000, DonHagVel2000, PieScaVel2018} and implemented~\cite{BanKay2013, BurCalKay2024, BurKayMou2024, KayMau2023, KayNoa2013} and (ii)~{\em arc}, or {\em switching time}, {\em parametrization techniques} earlier used for optimal control problems exhibiting discontinuous controls~\cite{KayNoa2003, MauBueKimKay2005}---also see \cite{AghHag2021,Hager2023} for more recent research in this area.

Problem~(OC) or Problem~(Pc) can be re-written in a concise form as
\[
\mbox{(Pgen)}\left\{\begin{array}{rl}
\ds\min_{u(\cdot)} &\ \ds b  \\[3mm]
\mbox{s.t.} &\ \dot{s}(t) = f(s(t),u(t))\,,\ \ s(0) = \overline{s}_0\,,\ s(t_f) = \overline{s}_f\,, \\[2mm] 
  &\ -b \le u(t) \le b\,,\ \ -a \le \kappa(t) \le a\,, \ \mbox{for a.e.\ }  t\in[0,t_f]\,,
\end{array}\right.
\]
where $s(t) := (x(t), y(t), \theta(t), \kappa(t))$ is the vector of state variables.  We note that, without any loss of generality, some components of the initial and terminal states may be opted to be not specified.

In what follows, we provide a review of the direct discretization and arc parametrization techniques as applied to Problem~(Pgen).  The descriptions of the way these techniques are implemented should in particular facilitate the {\em reproducibility} of the numerical experiments that we present in Section~\ref{sec:experiments}.

\subsection{Numerical methods}

We review here two numerical approaches to solving Problem~(OC).  The first approach, i.e., direct discretization of Problem~(OC) using the Euler scheme, results in a large-scale finite-dimensional optimization problem, the solution of which yields the structure of the solution, namely the number and order of the concatenation of the arcs, along with the places of the switchings/junctions at low precision.  Then the second approach, i.e., an arc parametrization technique, uses the known solution structure to find the junction, or switching, times at high precision.

\subsubsection{Direct discretization}
\label{sec:dir_discr}

We carry out discretization over a partition $\pi : \{t_0, t_1, t_2, \ldots, t_{N}\}$ of the time horizon $[0,t_f]$, such that
\[
0 = t_0 < t_1 < t_2 < \cdots < t_{N-1} < t_N = t_f\,,
\]
where $N$ is the number of discretization steps.  Over this partition, approximations of the state and control variables in Problem~(Pgen) are given by $s_i = (x_i, y_i, \theta_i, \kappa_i) \approx s(t_i)$, $i = 0,1,\ldots,N$, and $u_i\approx u(t_i)$, $i = 0,1,\ldots,N-1$. For simplicity, we take a regular partition $\pi$, namely define the step-size $h := t_f/N = t_{i+1} - t_i$, $i = 0,1,\ldots,N-1$.  Define $s_\pi := (s_0,s_1,\ldots,s_N)$ and  $u_\pi := (u_0,u_1,\ldots,u_{N-1})$, which approximate the functions $s(\cdot)$ and $u(\cdot)$, respectively.

\begin{remark}[Runge--Kutta discretization] \label{rem:discretization} \rm
By Theorem~\ref{theo:classification} and Remark~\ref{rem:types}, an optimal curve is a concatenation of various types of arcs; in other words, an optimal control variable is a piecewise-continuous function.  Therefore the state variable vector $s$ in Problem~(Pgen) above is of class $C^0$, deeming the Euler discretization scheme (which is an order-one and therefore the simplest Runge--Kutta method) to be more suitable than higher-order Runge--Kutta discretization schemes, in a direct discretization approach.  When we employed the trapezoidal rule, which is an order-two Runge--Kutta method, for example, we observed that not only the discrete solution obtained was not an (approximately) optimal solution (i.e., it did not numerically verify the maximum principle) in the case when the curvature was not constrained, but also, in the case when the curvature constraint was active, the (discrete) control variable exhibited severe {\em numerical chatter}.
\proofbox
\end{remark}

The Euler discretization of Problem~(Pgen) is posed as follows.
\[
\mbox{(P$_h$)}\left\{\begin{array}{rll}
\ds\min_{s_\pi,u_\pi,b} &\ \ds b & \\[3mm]
\mbox{s.t.} &\ s_{i+1} = s_i + h\,f(s_i,u_i)\,, & \ s_0 = \overline{s}_0\,,\ s_N = \overline{s}_f\,, \\[2mm] 
  &\ -b \le u_i \le b\,, & \ i =  0,1,\ldots,N - 1\,, \\[2mm]
  &\ -a \le \kappa_i \le a\,, & \ i =  0,1,\ldots,N\,.
\end{array}\right.
\]
For a typical optimal control problem, the partition size $N$ is large (usually in the order of thousands) in order to get a {\em reasonably accurate} solution; so, Problem~(P$_h$) constitutes a large-scale finite-dimensional optimization problem.  For finding, in the best case, a locally optimal solution of Problem~(P$_h$), various well-established general nonlinear programming software are available, such as Algencan~\cite{Andreani2007,BirMar2014}, which implements augmented Lagrangian techniques; Ipopt~\cite{WacBie2006}, which implements an interior point method; SNOPT~\cite{GilMurSau2005}, which implements a sequential quadratic programming algorithm; Knitro~\cite{Knitro}, which implements various interior point and active set algorithms to choose from.

\subsubsection{Arc parametrization}
\label{sec:arc_param}

We observe that, by \eqref{v}, $u(t)\in\{-b,0,b\}$, for $t\in[0,t_f]$.  Suppose that the optimal curve $(x(t),y(t))$ is a concatenation of $n_a$ arcs.  Then $u(t)$ has a jump at the {\em switching times} $t_k$, $k = 1,\ldots,n_a-1$, such that
\[
0 = t_0 < t_1 < \cdots < t_{n_a-1} < t_{n_a} = t_f\,.
\]
Along the $k$th arc, suppose that $u(t) =: \overline{u}_k$ for $t\in[t_{k-1},t_k]$, $k = 1,\ldots,n_a$, with $\overline{u}_k\in\{-b,0,b\}$, and define the state variable vector as $s^{(k)}(t) := s(t)$ for $t\in[t_{k-1},t_k]$.  Now the differential equations in Problem~(Pgen) can be written for the $k$th arc as
\begin{equation}  \label{eq:ODE_arc}
\dot{s}^{(k)}(t) = f(s^{(k)}(t),\overline{u}_k)\,,\ \mbox{for } t\in[t_{k-1},t_k]\,,\ k = 1,\ldots,n_a\,,
\end{equation}
along with the boundary conditions $s^{(1)}(0) = \overline{s}_0$, $s^{(n_a)}(t_f) = \overline{s}_f$, and the continuity conditions $s^{(k)}(t_{k-1}) = s^{(k-1)}(t_{k-1})$, $k = 2,\ldots,n_a$, for each of the $n_a$ ODEs in~\eqref{eq:ODE_arc}.  We can  interpret the system of $n_a$ ODEs given above with a slightly different viewpoint as follows: $s^{(1)}(0) = \overline{s}_0$ is the initial condition of the ODE in~\eqref{eq:ODE_arc} for $k = 1$, and the continuity  condition $s^{(k)}(t_{k-1}) = s^{(k-1)}(t_{k-1})$ constitutes the initial condition of the ODE in~\eqref{eq:ODE_arc} for $k = 2,\ldots,n_a$.

The $n_a$ ODEs described above can be spliced together by means of an {\em arc parametrization} as follows.  Let $\xi_k$ denote the length of the $k$th arc; namely that $\xi_k := t_k - t_{k-1}$, $k = 1,\ldots,n_a$.  As in ~\cite[Section~4.1]{KayNoa2003} or \cite[Section~4]{MauBueKimKay2005}, the interval $[t_{k-1},t_k]$ is mapped to the fixed interval $[(k-1)/n_a,k/n_a]$, and so $[0,t_f]$ is mapped to the fixed interval $[0,1]$.  Now the state equations in~\eqref{eq:ODE_arc} can be re-written as a single ODE in the state variable vector $s(t)$, parametrized w.r.t. the arc-lengths $\xi_k$, as
\begin{subequations}
	\begin{eqnarray}
\label{eq:ODE_arc_param}
\dot{s}(\tau) = n_a\,\xi_k\,f(s(\tau),\overline{u}_k)\,,\ \mbox{for } \tau\in[(k-1)/n_a,k/n_a]\,,\ k = 1,\ldots,n_a\,, \label{eq:ODE_arc_param_a} \\[1mm]
s(0) = \overline{s}_0\,,\ s(1) = \overline{s}_f\,, \label{eq:ODE_arc_param_b}
	\end{eqnarray}
\end{subequations}
where the continuity conditions are readily satisfied, as any solution $s(\cdot)$ of the ODE in~\eqref{eq:ODE_arc_param_a} is continuous over $[0,1]$.  Once the sequence $(\overline{u}_k)_{k=1}^{n_a}$ is provided (conceivably after solving Problem~(P$_h$)), the control constraints are no longer needed, and we have to solve the following problem to find an optimal bound $b$ and optimal arc lengths $(\xi_1,\ldots,\xi_{n_a}) =: \xi$, at a high precision.
\[
\mbox{(Pa)}\left\{\begin{array}{rl}
\ds\min_{b,\xi} &\ \ds b \\[3mm]
\mbox{s.t.} &\ \mbox{A high-order Runge--Kutta discretization of~\eqref{eq:ODE_arc_param_a}}\,, \\[2mm]
& \ s_0 = \overline{s}_0\,,\ s_N = \overline{s}_f\,, \\[2mm]
&\ -a \le \kappa_i \le a\,,\ \ i =  0,1,\ldots,N\,. 
\end{array}\right.
\]
In Problem~(Pa), $N$ is required to be an integer multiple of $n_a$.

\begin{remark}[High-order Runge--Kutta discretization] \rm
Detailed information on high-order\linebreak Runge--Kutta discretization schemes for optimal control problems can be found in \cite{BonLau2006,HaiLubWan2006,Kaya2010}.
\proofbox
\end{remark}

\subsection{Numerical experiments}
\label{sec:experiments}

Problems~(P$_h$) and (Pa) can be solved by the AMPL--Knitro suite.  AMPL \cite{AMPL} is an optimization modelling language, in which we employ Knitro~\cite{Knitro} (version 13.0.1 is used here).  For Problem~(P$_h$) we set $N = 2000$, the Knitro parameters {\tt alg=0} (meaning that it is left to Knitro to choose an appropriate algorithm), {\tt feastol=1e-12}, {\tt opttol=1e-12}, and for Problem~(Pa) we use the Runge--Kutta scheme, the sixth-order Gauss--Legendre discretization, and set $N = 400$ (or the nearest integer multiple of $n_a$ to $400$), {\tt alg=0}, {\tt feastol=1e-14}, {\tt opttol=1e-14}.

The Lagrange multipliers of the discretized ODE constraints in Problem~(P$_h$) in an AMPL code constitute nothing but the discretized adjoint variables of the optimal control problem~(Pgen) and they are provided by AMPL.  These variables facilitate a numerical verification of the necessary optimality conditions, namely the maximum principle.  For example, the knowledge about the switching function $\lambda_4$ allows one to check whether the control variable obtained via optimization verifies~\eqref{v} or not.  For the case when the curvature is constrained, it is also important to verify the complementarity conditions in~\eqref{eq:mu1}--\eqref{eq:mu2}, and this can be achieved by rearranging~\eqref{adjoint4} as
\begin{equation}  \label{eqn:mu1mu2}
\mu_1(t) - \mu_2(t) = \dot{\lambda}_4(t) + \lambda_3(t)\,,
\end{equation}
evaluating $\dot{\lambda}_4$ via a (forward) difference formula using $\lambda_4$ (provided by AMPL as explained above), substituting the value of $\lambda_3$, and making use of the fact that if $\mu_i(t) > 0$ then $\mu_j(t) = 0$ with $i\neq j$.

The graphs of the functions for both of the examples below are obtained by plotting their discrete approximations found by solving Problem~(P$_h$).  The (approximate) discrete solutions reported by the suite as ``locally optimal'' are referred to here as ``critical'', as we have also checked that they numerically (i.e., approximately) verify the necessary conditions of optimality furnished by the maximum principle.

\subsubsection{Example 1: unconstrained curvature}
\label{sec:ex1}

Consider Problem~(OC), equivalently Problem~(Pc) or Problem~(P), with
\[
(x_0,y_0,\theta_0,\kappa_0) = (0,0,-\pi/3,0)\quad\mbox{and}\quad
(x_f,y_f,\theta_f,\kappa_f) = (0.4,0.4,-\pi/6,0)\,,
\]
and large enough $a$ (or $a = \infty$) so that the curvature constraints never become active.  We recall that the oriented endpoints $(x_0,y_0,\theta_0) = (0,0,-\pi/3)$ and $(x_f,y_f,\theta_f) = (0.4,0.4,-\pi/6)$ were used for constructing curvature-constrained curves minimizing length in \cite{Kaya2017} and fixed-length curves minimizing the $L^\infty$-norm of curvature in \cite{KayNoaSch2024}.  Here, we fix the curve length to be $t_f = 2$.

Four critical curves of minimax spirality, i.e., curves satisfying (or numerically verifying) the maximum principle for Problem~(OC), that were encountered are depicted in Figure~\ref{fig:crit_curves_ex1}(a), by solving Problem~(P$_h$) employing direct discretization, using the AMPL--Knitro computational suite.  Plots of the control variable $u(t)$, which is the derivative of curvature, for each curve is shown in Figure~\ref{fig:crit_curves_ex1}(b).  We observe that the control is of bang--bang type, as expected by Theorem~\ref{theo:classification}, and (it turns out) that the control variable has three switchings for each of these four curves.  In other words, four spirals are concatenated to make any of these critical curves.

\begin{figure}[t!]
\begin{center}
\includegraphics[width=160mm]{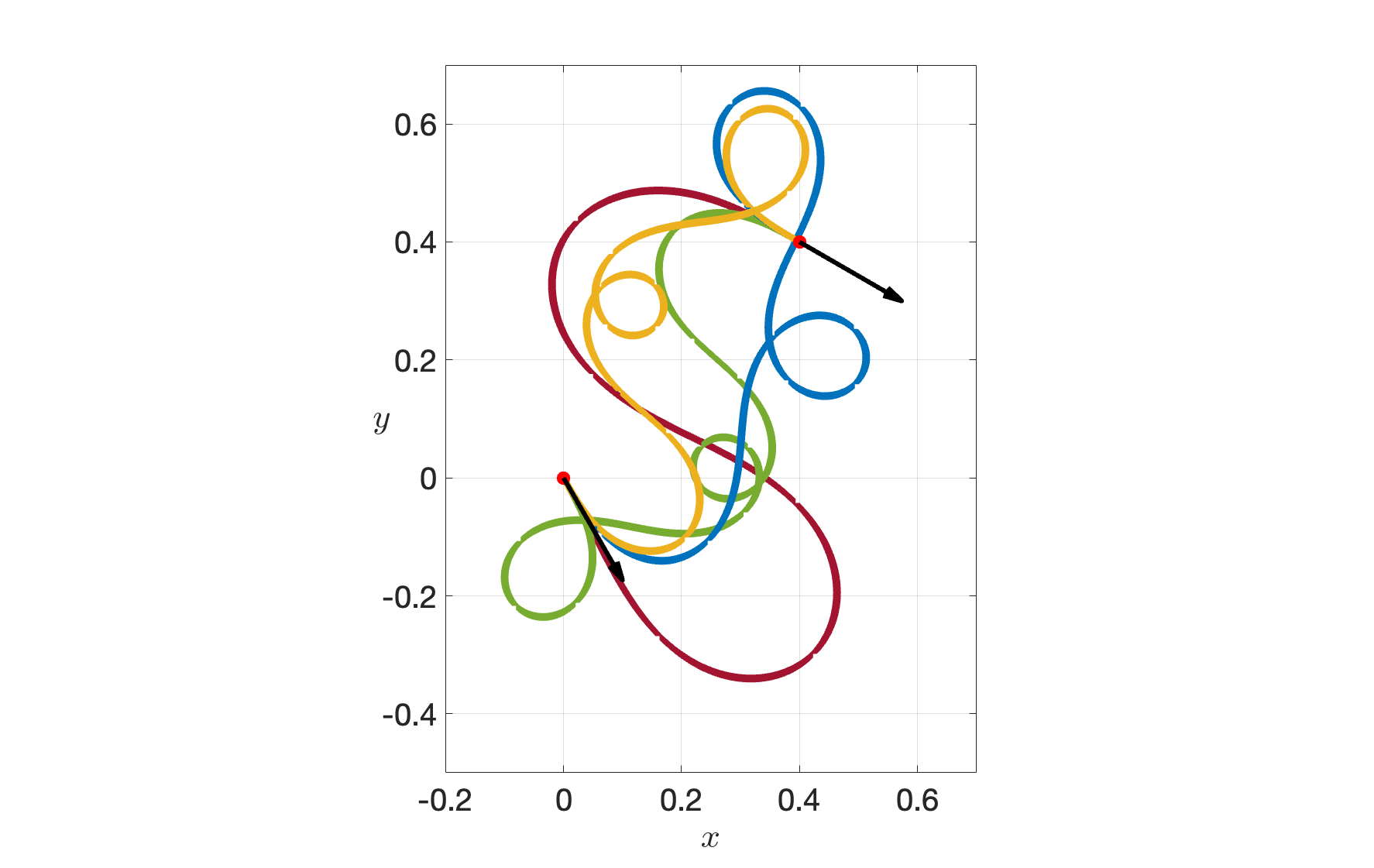} \\[0mm]
{\sf\small\hspace{7mm} (a) Critical curves of minimax spirality.}
\end{center}
\begin{center}
\includegraphics[width=145mm]{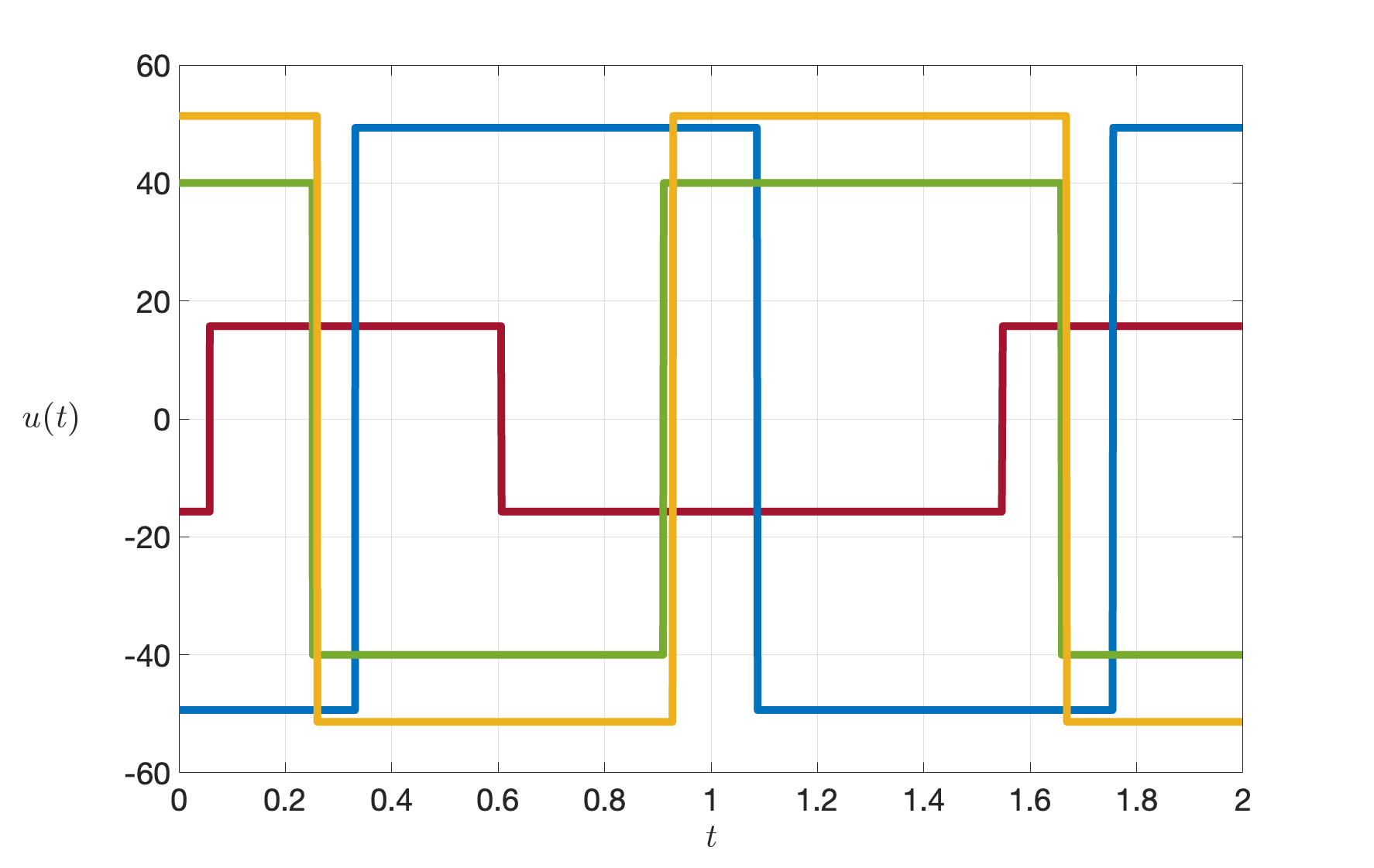} \\[0mm]
{\sf\small\hspace{7mm} (b) Control variable $u(t) = \dot{\kappa}(t)$.}
\end{center}
\
\caption{\sf Example 1---Critical curves of length $t_f = 2$ and of minimax spirality between the oriented points $(0,0,-\pi/3)$ and $(0.4,0.4,-\pi/6)$, and with $\kappa(0) = \kappa(2) = 0$.  Each curve satisfies the maximum principle and therefore is critical.  The curve in dark red has the smallest derivative of curvature, $b \approx 15.73$.}
\label{fig:crit_curves_ex1}
\end{figure}

\begin{figure}[t!]
\begin{center}
\includegraphics[width=140mm]{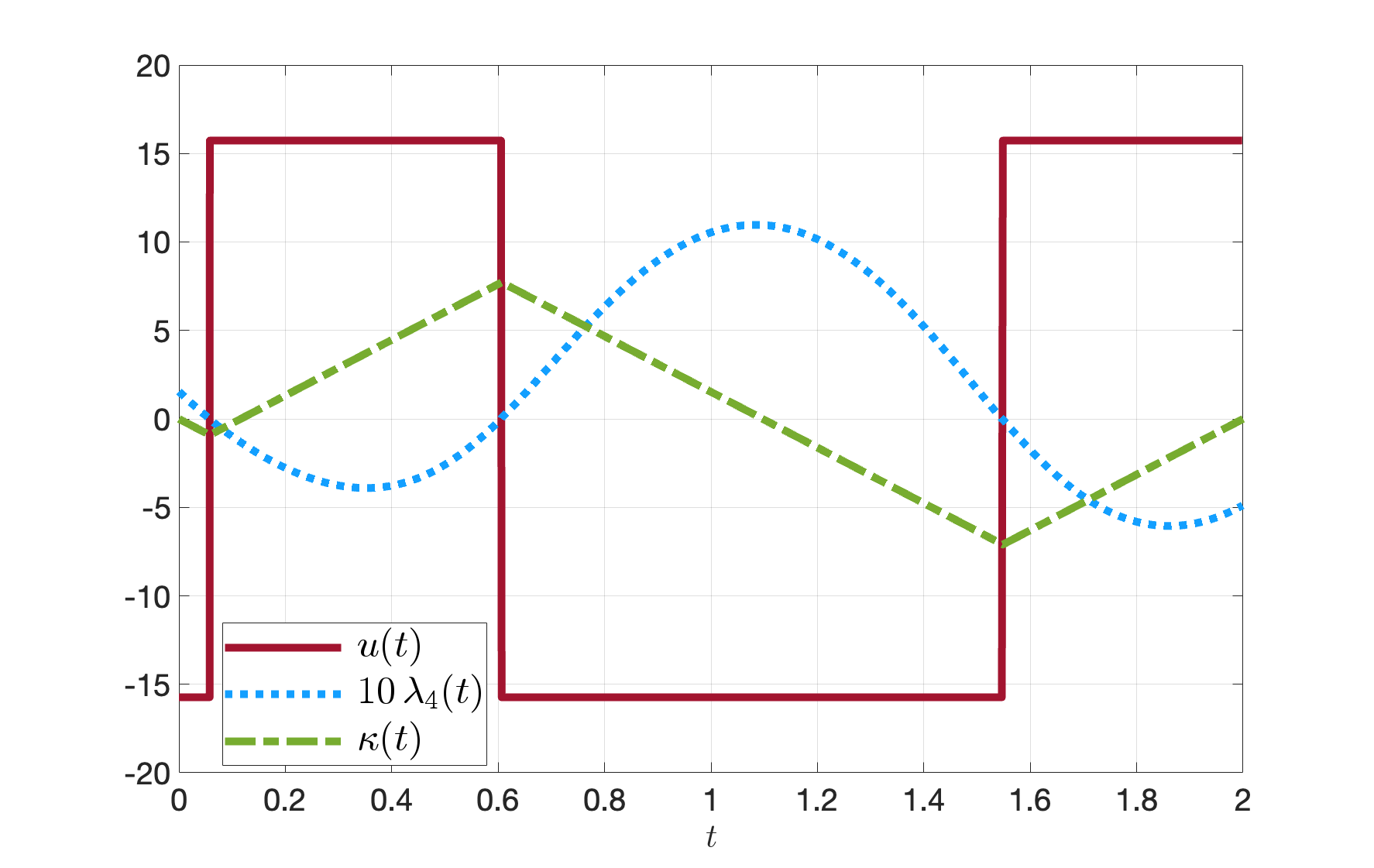}
\end{center}
\
\caption{\sf Example 1---``Best'' critical solution with verification of the necessary optimality condition that $u(t) = -b\,\sgn(\lambda_4(t))$.  Note that $\kappa(t)$ is affine in $t$ along each spiral arc.}
\label{fig:best_crit_curve_ex1}
\end{figure}

For the curves shown, using the solution structure given by the sequences $(\overline{u}_k)_{k=1}^4$, we used the arc parameterization approach and solved Problem~(Pa), by employing the AMPL--Knitro computational suite, to find the (optimal) bound $b$ on the derivative of curvature, in the order of increasing values, as $15.733062270883$, $40.016886269449$, $49.380682469500$ and $51.368649667030$, correct to 12~dp (decimal places).  Below we list the minimum bound on the curvature derivative, the arc lengths, and the switching times, of the best of these critical curves that was found as
\begin{eqnarray*}
b &=& 15.733062270883\,, \\
(\xi_1,\xi_2,\xi_3,\xi_4) &=& (0.058051025764,\ 0.548479899032,\ 0.941948974236,\ 0.451520100968)\,, \\
(t_1,t_2,t_3) &=& (0.058051025764,\ 0.606530924796,\ 1.548479899032)\,.
\end{eqnarray*}

The switching function of the best critical (solution) curve as well as the control variable is plotted in Figure~\ref{fig:best_crit_curve_ex1}.  The plot verifies numerically the optimality condition $u(t) = -b\,\sgn(\lambda_4(t))$ from~\eqref{v}, as well as the fact that the curvature $\kappa(t)$ is affine in $t$.

\subsubsection{Example 2: constrained curvature}
\label{sec:ex2}

Consider the same problem as in Example~1, except that the curvatures at the endpoints, namely $\kappa(0)$ and $\kappa(2)$, are free, i.e. that they are not specified, and that the (signed) curvature is constrained by imposing the lower and upper bounds of $-5$ and $5$, namely that $|\kappa(t)| \le 5$, for $t\in[0,t_f]$.  

After extensive numerical experiments, by solving Problem~(P$_h$), only one critical curve of minimax spirality was encountered, which is displayed in Figure~\ref{fig:crit_curve_ex2}. The structure of the solution is found as {\em bang--boundary--bang--boundary}, with three switching/junction times $t_1$, $t_2$ and $t_3$, and the sequence of constant control values $(\overline{u}_k)_{k=1}^4 = \{b,0,-b,0\}$, namely that
\[
u(t) := \left\{\begin{array}{rl}
b\,, & \ \mbox{if } 0 \le t < t_1\,, \\
0\,, & \ \mbox{if } t_1 \le t < t_2\,, \\
-b\,, & \ \mbox{if } t_2 \le t < t_3\,, \\
0\,, & \ \mbox{if } t_3 \le t \le 2\,. \\
\end{array} \right.
\]

\begin{figure}[t!]
\begin{center}
\includegraphics[width=160mm]{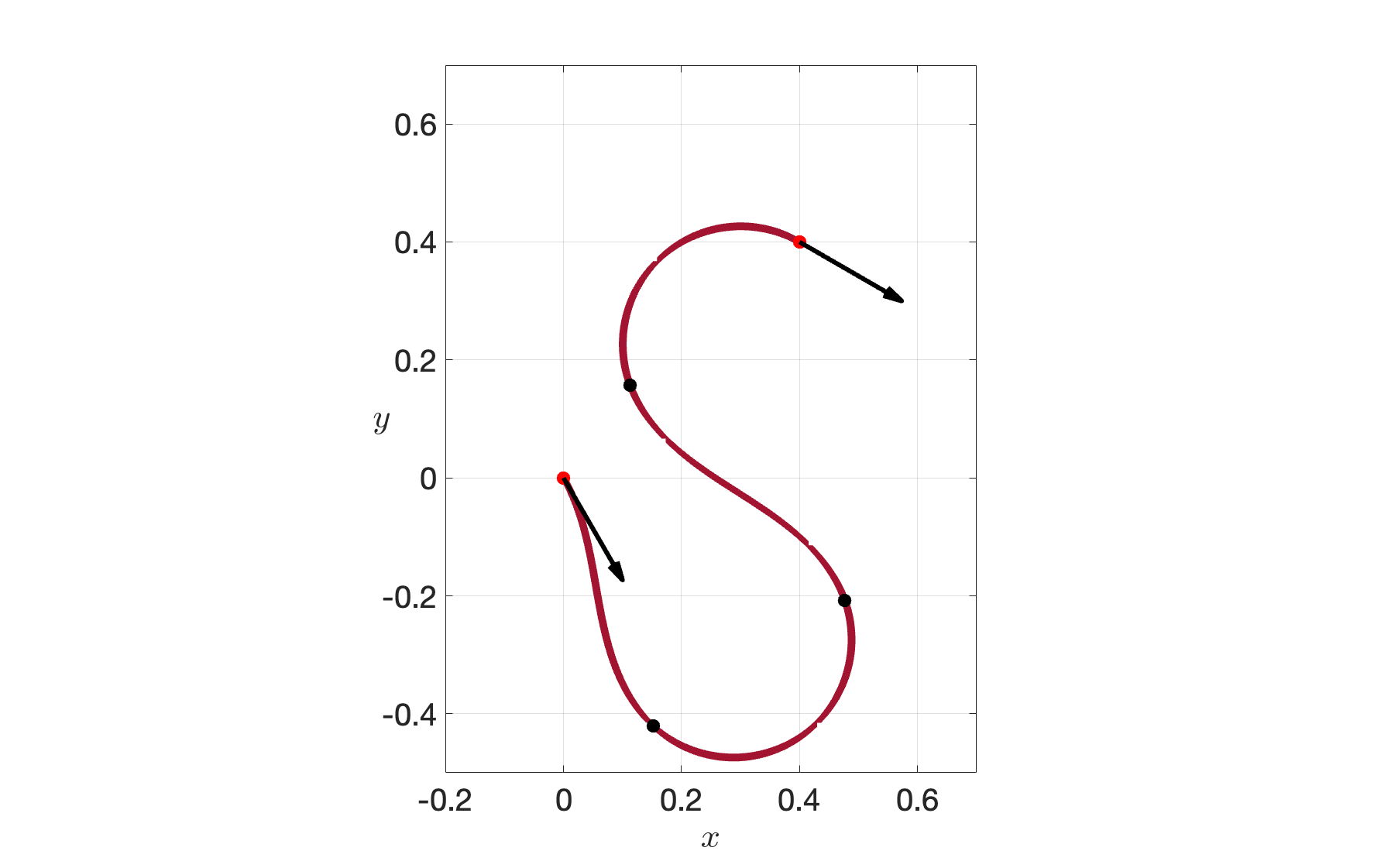}
\end{center}\
\caption{\sf Example 2---Critical curve of length $t_f = 2$ and of minimax spirality between the oriented points $(0,0,-\pi/3)$ and $(0.4,0.4,-\pi/6)$, with constraint $|\kappa(t)| \le 5$.  The three junctions/switching points are indicated by black dots.  One gets $b\approx 19.01$.}
\label{fig:crit_curve_ex2}
\end{figure}

Using the solution structure given above, a numerical solution of the arc-parametrizing problem~(Pa) yields the high-precision values of the optimal curvature bound, arc lengths and switching/junction times as follows.
\begin{eqnarray*}
b &=& 19.012850374851\,, \\
(\xi_1,\xi_2,\xi_3,\xi_4) &=& (0.454980338573,\ 0.531189265997,\ 0.525960064001,\ 0.487870331429)\,, \\
(t_1,t_2,t_3) &=& (0.454980338573,\ 0.986169604570,\ 1.512129668571)\,.
\end{eqnarray*}

The graphs of the switching function $\lambda_4(t)$ (in fact, a scalar multiple of $\lambda_4(t)$ for easier viewing), the control variable $u(t)$ and the curvature $\kappa(t)$ are plotted in Figure~\ref{fig:crit_curve_optimality_ex2}(a).  It is readily observed from the plots that the optimality condition $u(t) = -b\,\sgn(\lambda_4(t))$ is verified.  It is also observed that when the curvature constraint $|\kappa(t)| \le 5$ is active, that is when $t_1\le t < t_2$ or $t_3 \le t \le 2$, one has that $\lambda_4(t) = 0$ and $u(t) = 0$. It is also verified that $\kappa(t)$ is affine in $t$ along each spiral arc.  The graphs of the constraint multiplier functions $\mu_1(t)$ and $\mu_2(t)$ are plotted in Figure~\ref{fig:crit_curve_optimality_ex2}(b) by using \eqref{eqn:mu1mu2} and its ensuing explanation.  Finally, the complementarity conditions \eqref{eq:mu1}--\eqref{eq:mu2} are verified by these plots.

\begin{figure}[t!]
\begin{center}
\includegraphics[width=160mm]{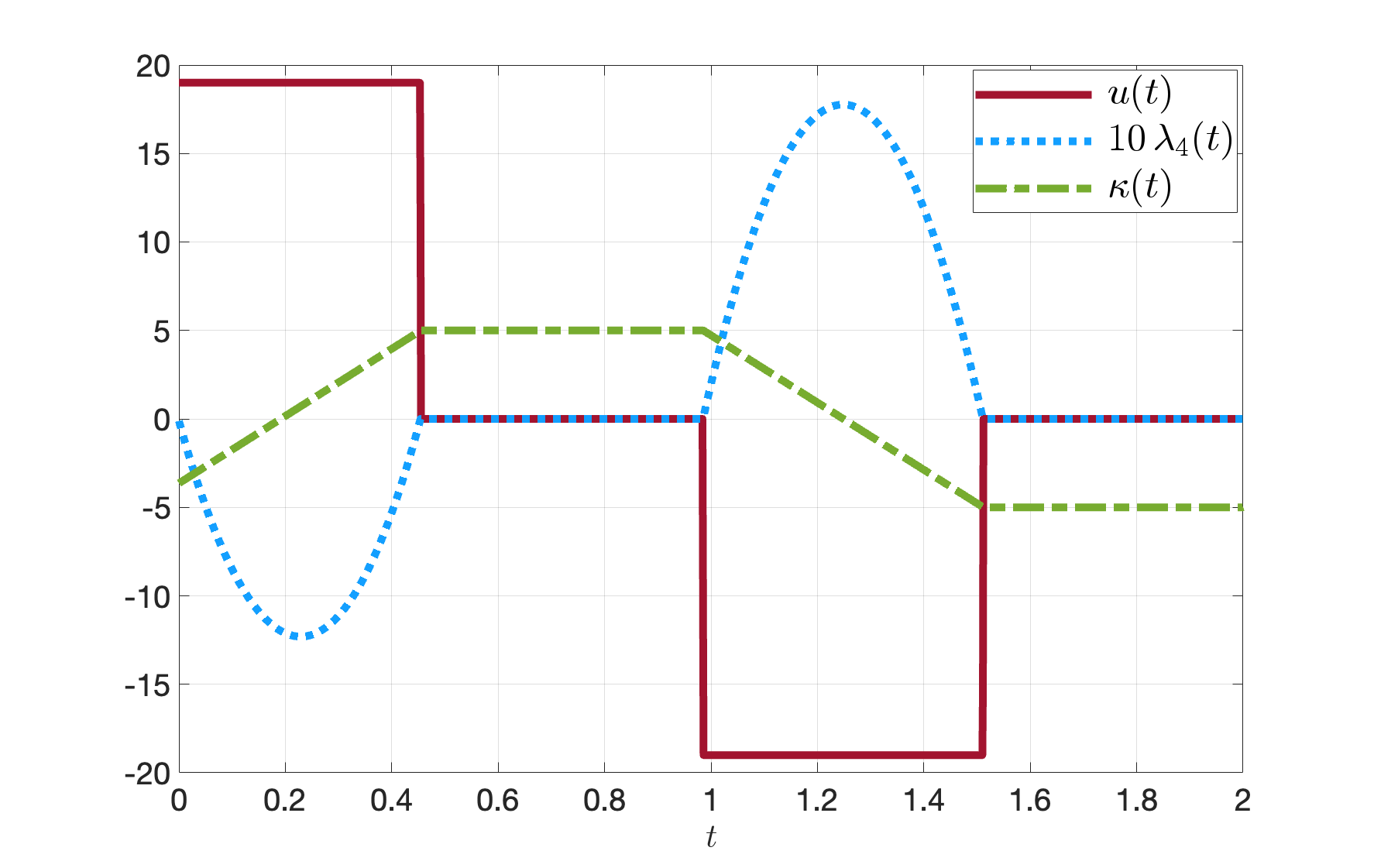} \\[0mm]
{\sf\small\hspace{7mm} (a) Verification of (i) $u(t) = -b\,\sgn(\lambda_4(t))$, (ii) $\lambda_4 = 0$ and $u(t) = 0$ along an active curvature constraint and (iii) $\kappa(t)$ is affine in $t$ along each spiral arc.}
\end{center}
\begin{center}
\includegraphics[width=160mm]{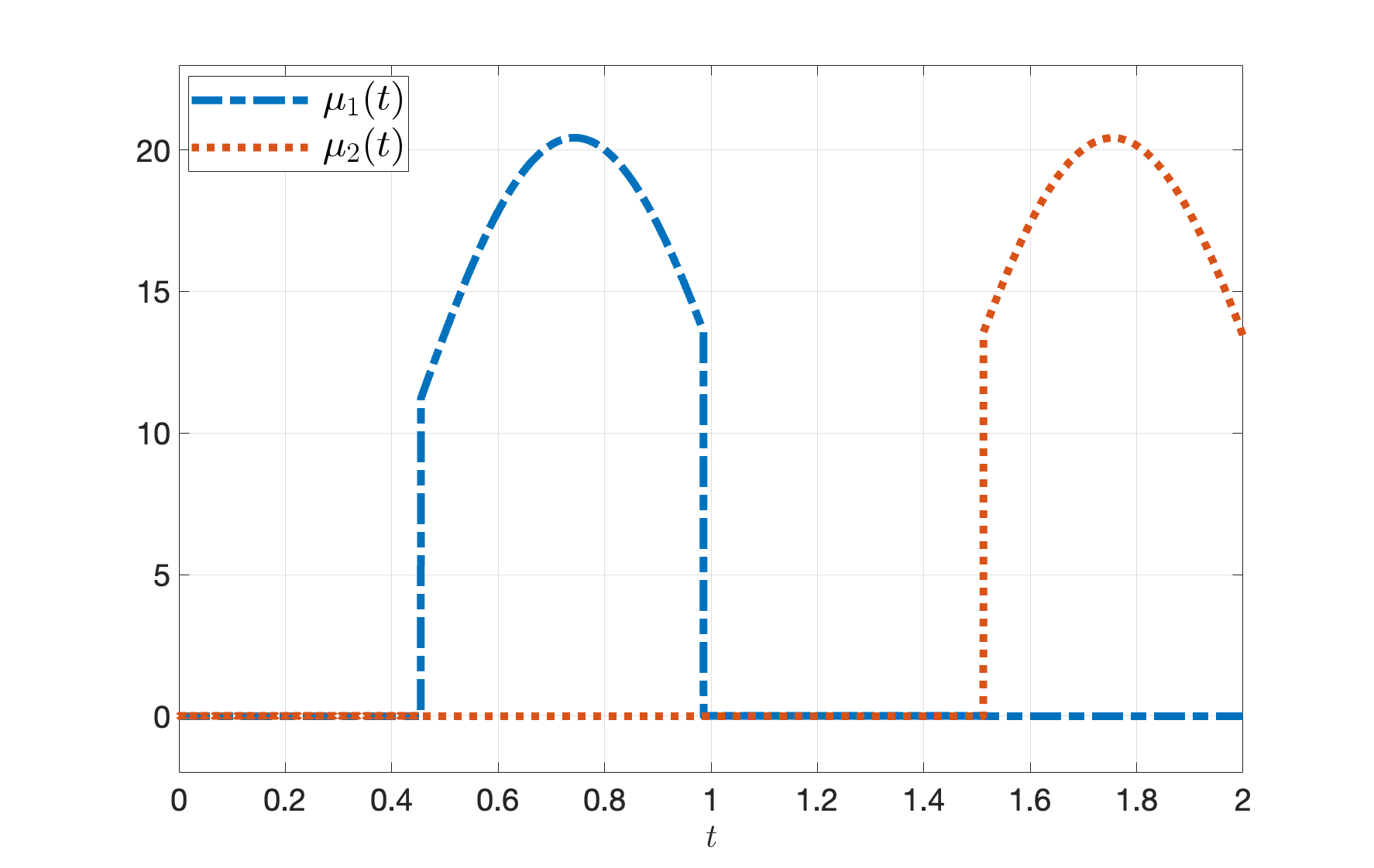} \\[0mm]
{\sf\small\hspace{7mm} (b) Verification of the complementarity conditions \eqref{eq:mu1}--\eqref{eq:mu2}.}
\end{center}
\
\caption{\sf Example 2---Numerical verification of the necessary optimality conditions.}
\label{fig:crit_curve_optimality_ex2}
\end{figure}

\subsubsection{Example 3: joining two circular arcs}
\label{sec:ex2}

Consider a modification of the problem in Example~1, with
\[
(x_0,y_0,\theta_0,\kappa_0) = (0,0,\pi/3,5)\quad\mbox{and}\quad
(x_f,y_f,\theta_f,\kappa_f) = (0.4,0.4,\pi/4,2)\,,
\]
the length $t_f = 0.6$, and and no path constraint is imposed on the curvature.  One can imagine here that one would like to join two circular arcs, one circular arc with radius $0.2$ at one endpoint and another one with radius $0.5$ at the other endpoint---see Figure~\ref{fig:crit_curves_ex3a}(a) for a visualization.  We refer to this first instance as Example~3a.

The solution curve, obtained by solving Problem~(P$_h$), is depicted in Figure~\ref{fig:crit_curves_ex3a}(a), and a plot of the control variable is shown in Figure~\ref{fig:crit_curves_ex3a}(b).  As can be seen from the plot, the control function has four switchings; so, the solution curve is a concatenation of five spirals.  The points where these spirals are concatenated are displayed by four black dots in Figure~\ref{fig:crit_curves_ex3a}(a).  

\begin{figure}[t!]
\begin{center}
\includegraphics[width=160mm]{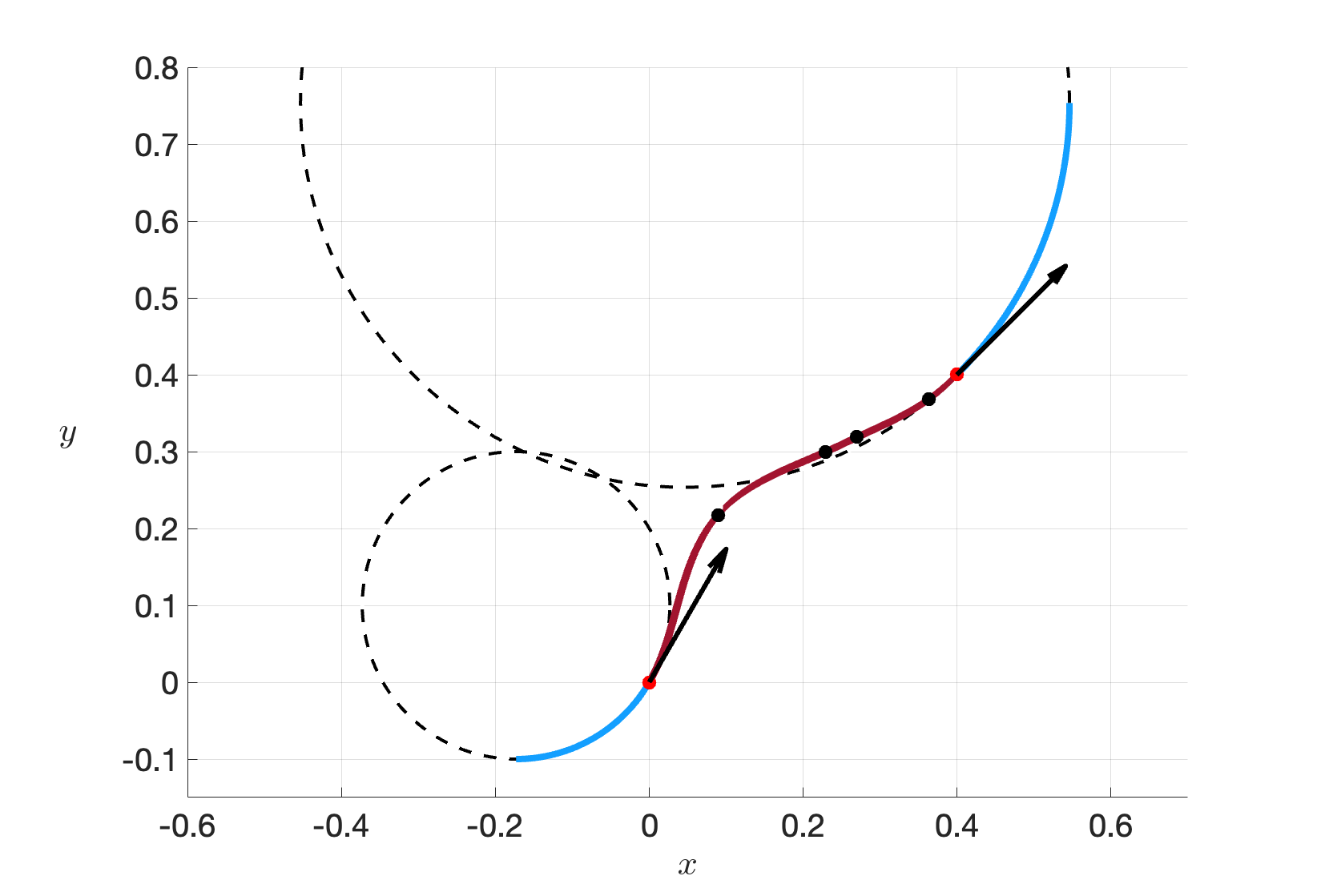} \\[0mm]
{\sf\small\hspace{7mm} (a) Critical curve (in dark red) consisting of five spirals concatenated at the black dots and joining two circular arcs (in light blue).}
\end{center}
\begin{center}
\includegraphics[width=145mm]{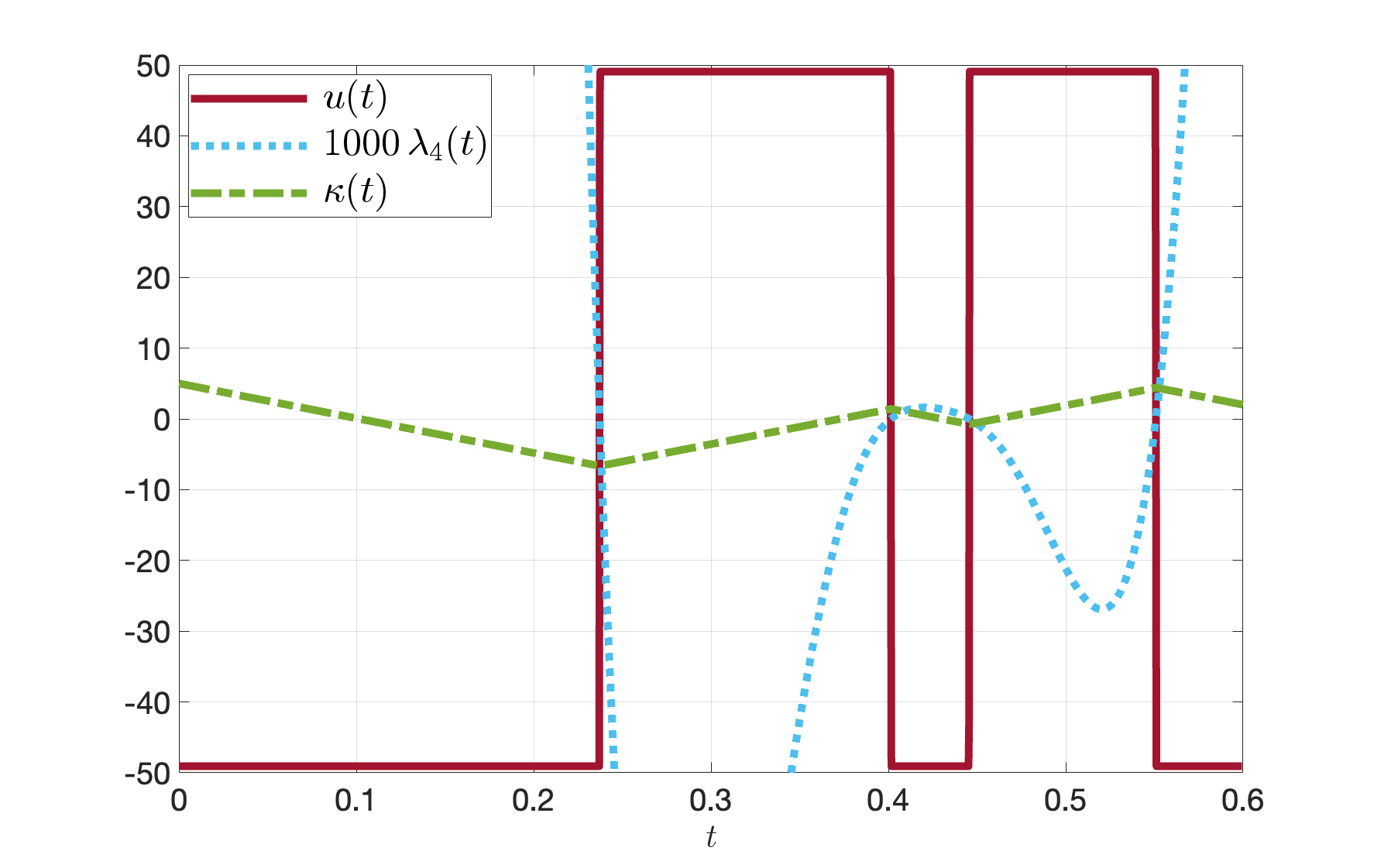} \\[0mm]
{\sf\small\hspace{7mm} (b) Verification of (i) $u(t) = -b\,\sgn(\lambda_4(t))$ and (ii) $\kappa(t)$ is affine in $t$ along each spiral arc.}
\end{center}
\
\caption{\sf Example 3a---Critical curve of length $t_f = 0.6$ and of minimax spirality between the oriented points $(0,0,\pi/3)$ and $(0.4,0.4,\pi/4)$ joining two circular arcs, with $\kappa(0) = 5$ and $\kappa(0.6) = 2$.  Here one gets
 $b\approx 49.0$.}
\label{fig:crit_curves_ex3a}
\end{figure}

From the graphs of the switching function $\lambda_4(t)$ (multiplied by 1000 for easier viewing) and the curvature $\kappa(t)$, one can see that the optimality conditions $u(t) = -b\,\sgn(\lambda_4(t))$ from~\eqref{v}, and the fact that the curvature $\kappa(t)$ is affine in $t$, are numerically verified.

Utilizing the control structure (five spirals with $u(0) = -b$) obtained by solving Problem~(P$_h$), the minimum bound $b$ on the curvature derivative, the arc lengths, and the switching times, are computed at a high precision, by solving Problem~(Pa), as
{\small
\begin{eqnarray*}
b &=& 48.985303304067\,, \\
(\xi_1,\ldots,\xi_5) &=& (0.237659563282,\ 0.164288710499,\ 0.043943296148,\ 0.105089860239,\ 0.049018569832)\,, \\
(t_1,\ldots,t_4) &=& (0.237659563282,\ 0.401948273781,\ 0.445891569929,\ 0.550981430168)\,.
\end{eqnarray*}}

As pointed earlier, Theorem~\ref{theo:classification} does not inform about the number of spirals one should concatenate.  The critical solutions we found in Example~1 were formed by a concatenation of four arcs.  In the current example, however, the number of spirals in the critical curve is found to be five (i.e., four switchings).  Next we will numerically illustrate that, with a slight modification of the problem data, one can observe evidence for possibly infinitely many switchings, i.e., the chatter as in Fuller's phenomenon~\cite{Borisov2000}.

Consider the same problem with $\kappa(0) = 8$, instead of $\kappa(0) = 5$, which is the second instance referred to as Example~3b.  The solution of the directly discretized problem~(P$_h$) with this modified end curvature (and $N = 2000$), which results in $b\approx 90.0$ (correct to one dp), is plotted in Figure~\ref{fig:crit_curves_ex3b}.  The solution curve in Figure~\ref{fig:crit_curves_ex3b}(a) seems to be a concatenation of spirals, by also looking at the graph of $u(t)$ in Figure~\ref{fig:crit_curves_ex3b}(b) which appears to be bang--bang except that the plot exhibits numerical chatter in the ``fourth arc'' roughly between $t = 0.35$ and $t = 0.48$.  Although the graph of $\lambda_4(t)$ numerically verifies the necessary optimality condition $u(t) = -b\,\sgn(\lambda_4(t))$ along the ``regular'' bang-arcs, we observe that $\lambda_4(t) \approx 0$ along the fourth ``chattering arc''.  At this point we refer to Lemma~\ref{lem:tot_singular} according to which partial singularity is not possible if there is no chatter.  

We point out that optimal control chatter has also been numerically observed in problems involving other dynamical systems, for example, in the minimum-time control of an underwater vehicle~\cite{ChySusMauVos2004}.  In~\cite{ZhuTreCer2016}, where the minimum time planar tilting maneuver of a spacecraft is studied, it is proved that there exist optimal chattering arcs when a singular junction occurs.  In the current paper, we restrict our attention to a general classification of the curves of minimax spirality and their computation.  However, it would be interesting to prove, if it is possible, the existence (and features) of chattering for the problem we are studying, as part of future work.

\begin{figure}[ht!]
\begin{center}
\includegraphics[width=160mm]{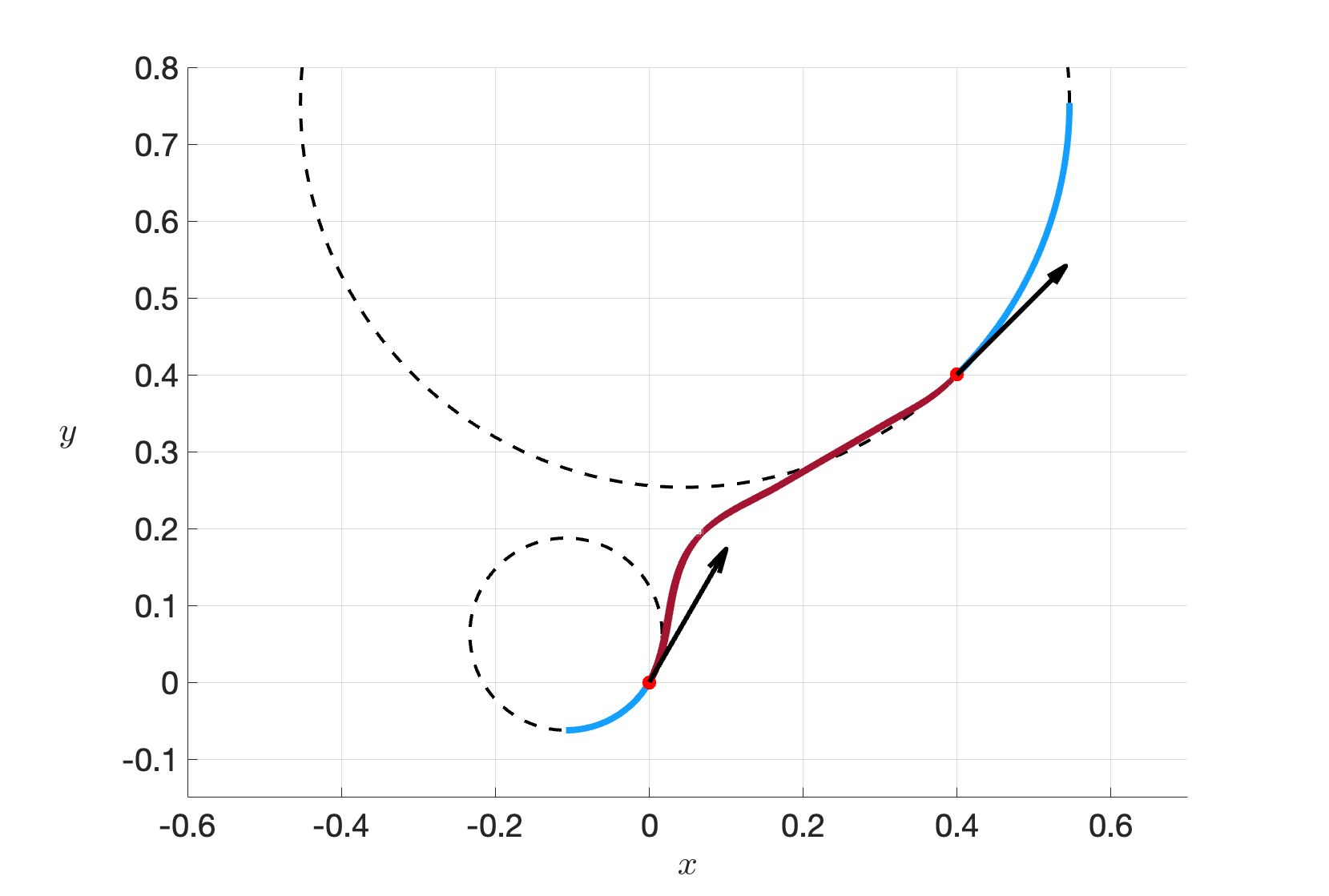} \\[0mm]
{\sf\small\hspace{7mm} (a) Two circular arcs (in light blue) are joined by a curve (in dark red) which solves Problem~(P$_h$).  It is not clear how many arcs are concatenated.}
\end{center}
\begin{center}
\includegraphics[width=145mm]{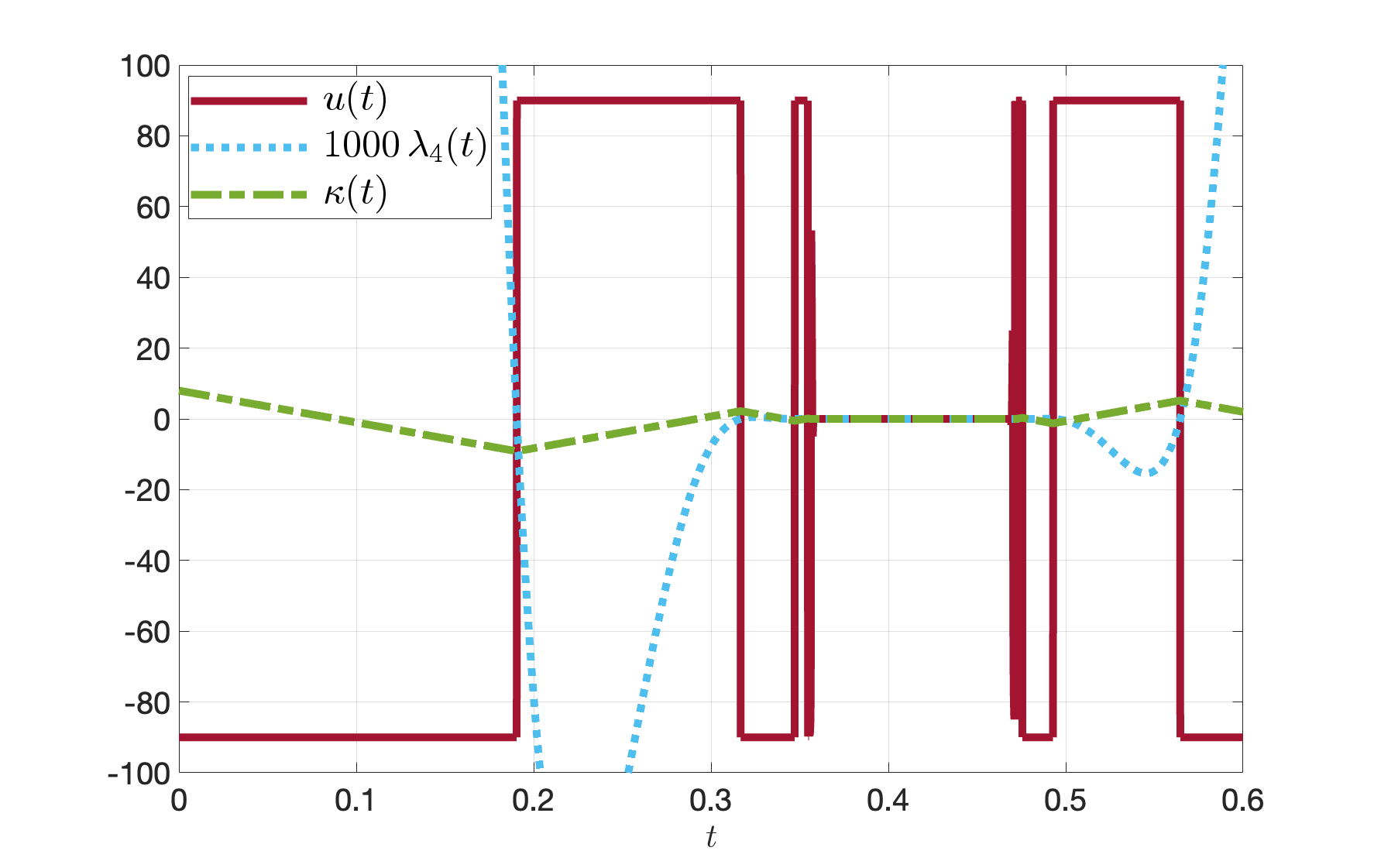} \\[0mm]
{\sf\small\hspace{7mm} (b) Partial verification of (i) $u(t) = -b\,\sgn(\lambda_4(t))$ and (ii) $\kappa(t)$ is affine in $t$ along spiral arcs.  The fourth arc exhibits (evidence of) chatter with $\lambda_4(t)\approx 0$ along the arc.}
\end{center}
\
\caption{\sf Example 3b---Curve of length $t_f = 0.6$ solving Problem~(P$_h$) between the oriented points $(0,0,\pi/3)$ and $(0.4,0.4,\pi/4)$ joining two circular arcs, with $\kappa(0) = 8$ and $\kappa(0.6) = 2$.  Here one gets $b\approx 90.0$.}
\label{fig:crit_curves_ex3b}
\end{figure}

Obviously, a large number of switchings (let alone infinitely many switchings) cannot be implemented. However, the graph of $u(t)$ in Figure~\ref{fig:crit_curves_ex3b}(b) informs one as to what one can do in a practical setting: the chattering (fourth) arc might be replaced by a straight line, i.e., one can think of assigning $u(t) = 0$ along the arc, just by graphically observing from Figure~\ref{fig:crit_curves_ex3b}(b) that not only $u(t) = 0$ most of the time, but also that $\lambda(t) = \kappa(t) = 0$, along the same arc.  Then the structure of the solution can be described by the sequence of constant control values $(\overline{u}_k)_{k=1}^7 = \{-b,b,-b,0,-b,b,-b\}$; in other words,
\[
u(t) := \left\{\begin{array}{rl}
-b\,, & \ \mbox{if } 0 \le t < t_1\,, \\
b\,, & \ \mbox{if } t_1 \le t < t_2\,, \\
-b\,, & \ \mbox{if } t_2 \le t < t_3\,, \\
0\,, & \ \mbox{if } t_3 \le t \le t_4\,, \\
-b\,, & \ \mbox{if } t_4 \le t < t_5\,, \\
b\,, & \ \mbox{if } t_5 \le t < t_6\,, \\
-b\,, & \ \mbox{if } t_6 \le t < 0.6\,, \\
\end{array} \right.
\]
where $t_i$, $i = 1,\ldots,6$, are the switching times.  This parameterization is referred to as Example~3c.  When this structure is fed into Problem~(Pa), the solution depicted in Figure~\ref{fig:crit_curves_ex3c} is obtained.  To the naked eye, there does not seem to be any difference between the curves in Figures~\ref{fig:crit_curves_ex3b}(a) and \ref{fig:crit_curves_ex3c}(a).  This is of course true only up to a precision.

\begin{figure}[t!]
\begin{center}
\includegraphics[width=160mm]{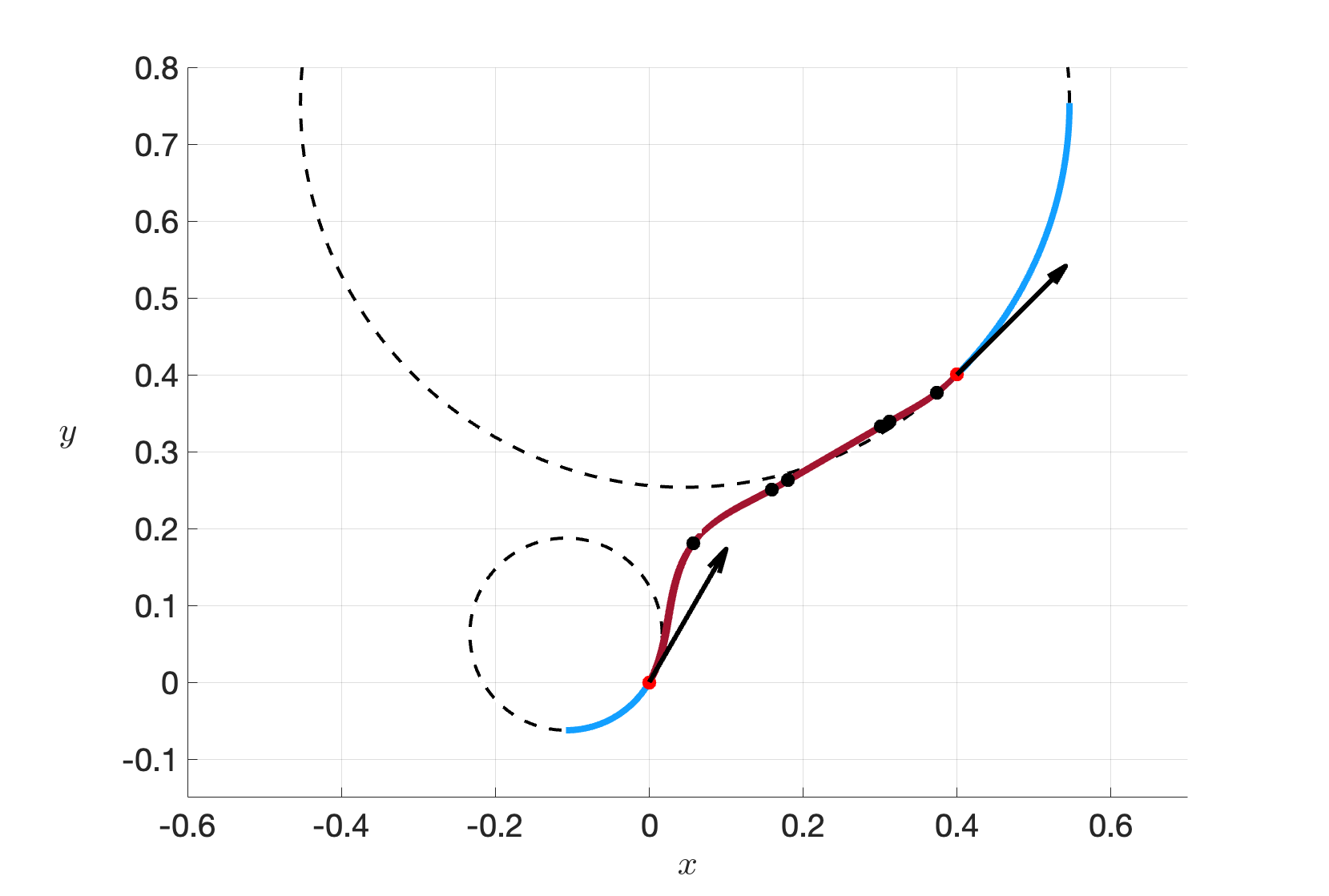} \\[0mm]
{\sf\small\hspace{7mm} (a) Two circular arcs (in light blue) are joined by a curve (in dark red) which solves Problem~(Pa).  Black dots indicate the points where the arcs are concatenated. The fourth arc is a straight line.}
\end{center}
\begin{center}
\includegraphics[width=145mm]{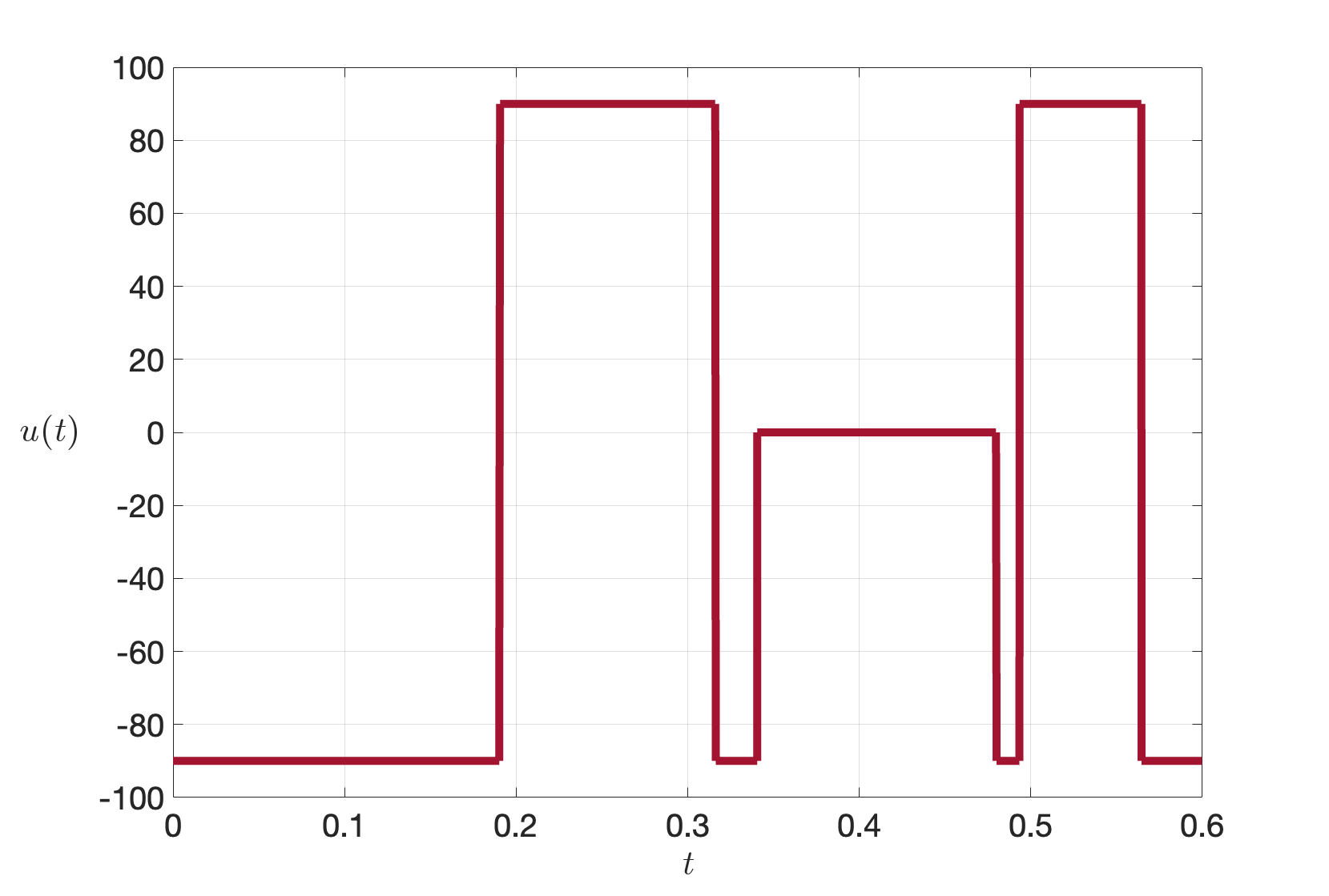} \\[0mm]
{\sf\small\hspace{7mm} (b) The control variable.}
\end{center}
\
\caption{\sf Example 3c---Curve of length $t_f = 0.6$ solving Problem~(Pa) between the oriented points $(0,0,\pi/3)$ and $(0.4,0.4,\pi/4)$ joining two circular arcs, with $\kappa(0) = 8$ and $\kappa(0.6) = 2$. Here one gets $b\approx 89.95$.}
\label{fig:crit_curves_ex3c}
\end{figure}

The solution of Problem~(Pa) yields the following high precision results.
{\small
\begin{eqnarray*}
b &=& 89.945849353595\,, \\
(\xi_1,\ldots,\xi_7) &=& (0.190593874793,\ 0.125757754134,\ 0.024206619117,\ 0.139477144951,\ \\
&&\hspace*{2mm} 0.013358892816,\ 0.071150272188,\ 0.035455442002)\,, \\
(t_1,\dots,t_6) &=& (0.190593874793,\ 0.316351628926,\ 0.340558248044,\ 0.480035392994,\ \\
&&\hspace*{2mm} 0.493394285810,\ 0.564544557998)\,.
\end{eqnarray*}}

Figure~\ref{fig:crit_curves_ex3c}(a) also depicts by black dots the points at which the seven arcs are concatenated.  We re-iterate that the resulting curve is the track sum of three spirals, one straight line segment and three more spirals.

\section{Conclusion}
\label{sec:conclusion}

We have studied the problem of finding curves of minimum spirality with and without a constraint on curvature and presented results classifying the types of solution curves.  We also proposed numerical approaches to compute these curves.

As mentioned earlier in the Introduction, minimum spirality is used as a measure of (maximum) comfort in the design of railway tracks and roads.  It seems to be common practice to connect circular and/or straight tracks by a single easement spiral.  In this paper, we show that for optimality more than just one spiral needs to be concatenated.  We also show that under a curvature constraint (simple bounds on curvature) the optimal track will in general be a concatenation of spirals and circular arcs (and straight lines).

We obtain what looks like numerical evidence that the number of arcs required to be concatenated in an optimal curve can be infinite (Fuller's phenomenon).  We propose a numerical remedy for the example instance of concern, at least to obtain an approximately optimal solution. It would be interesting to investigate and see if it is possible to prove Fuller's phenomenon for the problem studied here, in a way similar to that done for a different problem in~\cite{ZhuTreCer2016} or that geometrically described in the earlier review in~\cite{Borisov2000}.

Again, in practical applications, it might be necessary to minimize not only the spirality but also the length of a curve.  This is a question of multi-objective optimization, which would indeed be interesting to look at in a similar fashion as in~\cite{KayNoaZha2024}.

\section*{Acknowledgments}
The authors would like to thank the two anonymous reviewers whose insightful comments improved the manuscript.

\end{document}